\documentclass[11pt]{amsart}
\usepackage{graphicx,subfigure,epsfig}
\usepackage{mathtools}
\DeclarePairedDelimiter\abs{\lvert}{\rvert}

\newtheorem{theorem}{Theorem}[section]
\newtheorem{corollary}[theorem]{Corollary}
\newtheorem{remark}[theorem]{Remark}
\newtheorem{proposition}[theorem]{Proposition}
\newtheorem{lemma}[theorem]{Lemma}

\newtheorem{definition}[theorem]{Definition}
\newtheorem{example}[theorem]{Example}
\newtheorem{question}[theorem]{Question}

\begin{document}

\title[Two-variable polynomial invariants of virtual knots]{Two-variable polynomial invariants of virtual knots  arising from flat virtual \\ knot invariants}

\author{Kirandeep KAUR}
\address{Department of Mathematics, Indian Institute of Technology Ropar, India}   
\email{kirandeep.kaur@iitrpr.ac.in} 

\author{Madeti PRABHAKAR}
\address{Department of Mathematics, Indian Institute of Technology Ropar, India}   
\email{prabhakar@iitrpr.ac.in}

\author{Andrei VESNIN}
\address{Tomsk State University, Tomsk, 634050, Russia \\ and Sobolev Institute of Mathematics, Novosibirsk, 630090, Russia} 
\email{vesnin@math.nsc.ru} 

\thanks{
The first and second named authors were supported by DST -- RSF Project INT/RUS/RSF/P-2. The third named author was supported by the Russian Science Foundation (grant no. 16-41-02006).
}



\begin{abstract}
We introduce two sequences of two-variable polynomials $\{ L^n_K (t, \ell)\}_{n=1}^{\infty}$ and $\{ F^n_K (t, \ell)\}_{n=1}^{\infty}$,  expressed in terms of index value of a crossing and $n$-dwrithe value of a virtual knot $K$, where $t$ and $\ell$ are variables. Basing on the fact that $n$-dwrithe is a flat virtual knot invariant we prove that $L^n_K$ and $F^n_K$ are virtual knot invariants containing Kauffman affine index polynomial as a particular case.  Using $L^n_K$ we give sufficient conditions when virtual knot does not admit cosmetic crossing change. 
\end{abstract}

\maketitle 

%%%

\section{Introduction}

Virtual knots were  introduced by L.~Kauffman \cite{kauffman1999virtual} as a generalization of classical knots and presented by virtual knot diagrams having classical crossings as well as virtual crossings. Equivalence  between two virtual knot diagrams can be determined through classical Reidemeister moves and virtual Reidemeister moves shown in Fig.~\ref{fig1a} and Fig.~\ref{fig1b}, respectively. 
\begin{figure}[!ht] 
\centering
\subfigure[Classical Reidemeister moves.]
{\includegraphics[scale=0.44]{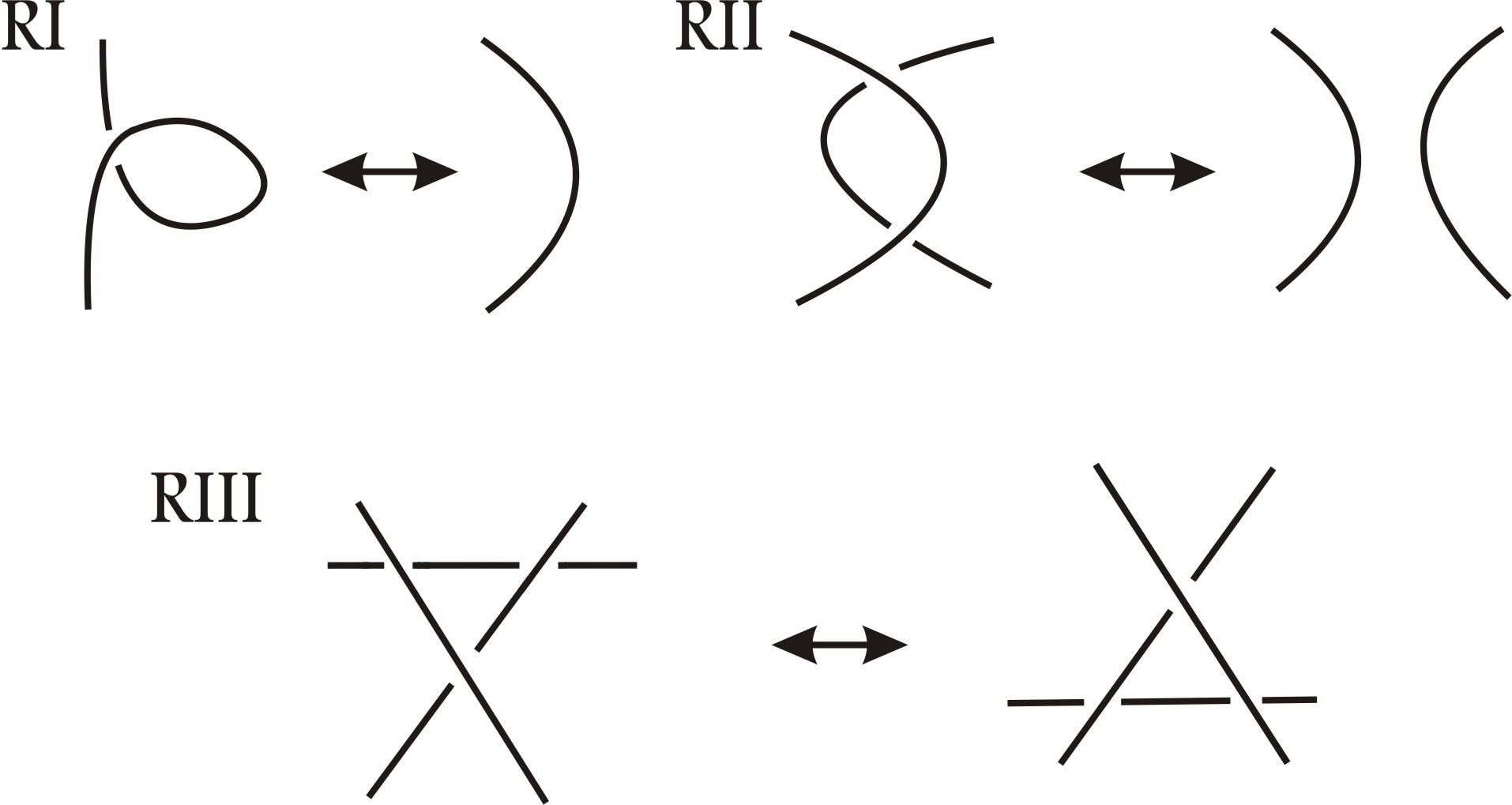} \label{fig1a}
} \medskip \subfigure[Virtual Reidemeister moves.]
{
\includegraphics[scale=.44]{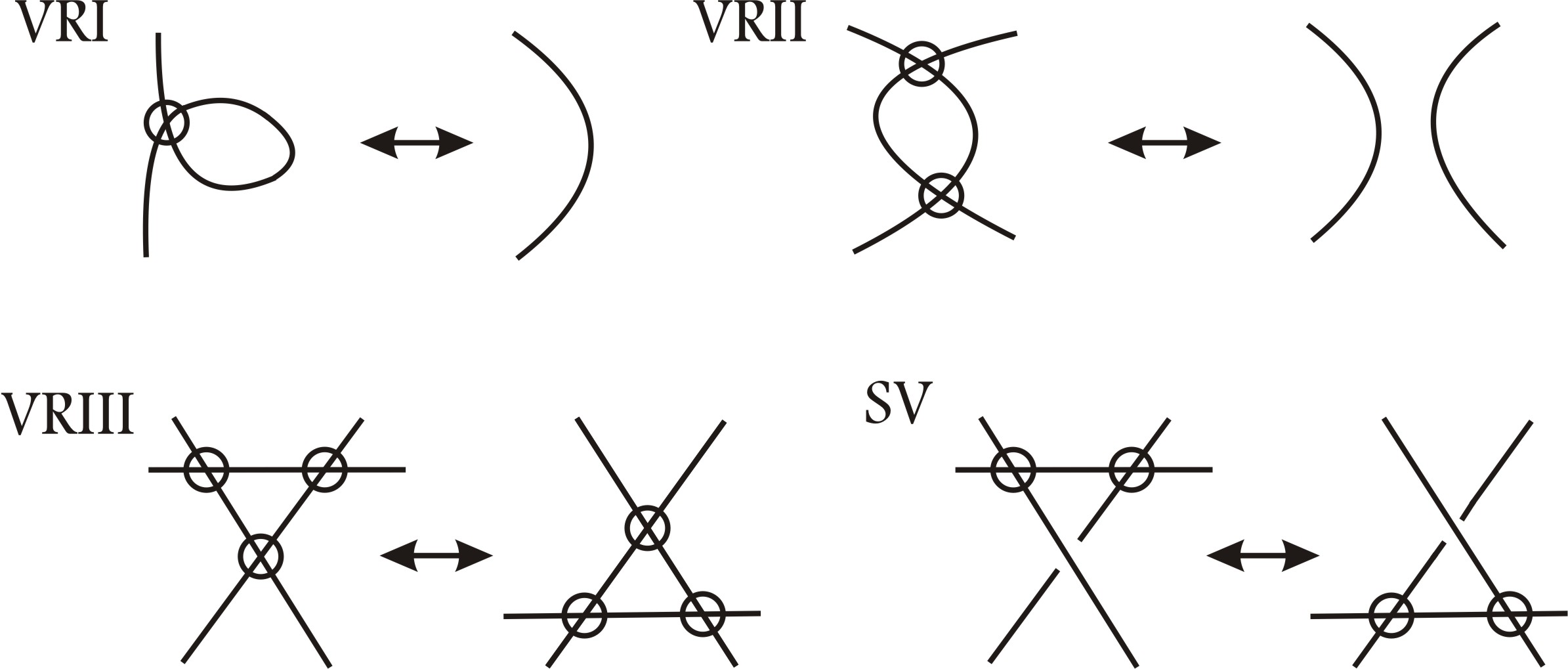} \label{fig1b}
}
\caption{Reidemeister moves.} \label{fig-rei}
\end{figure} 

Various invariants are known to distinguish two virtual knots. We are mainly interested in invariants of polynomial type. In the recent years, many polynomial invariants of virtual knots and links have been introduced. Among them are affine index polynomial by L~Kauffman~\cite{kauffman2013affine}, writhe polynomial by Z.~Cheng and H.~Cao~\cite{cheng2013polynomial}, wriggle polynomial by L.~Folwaczny and L.~Kauffman~\cite{folwaczny2013linking}, arrow polynomial by H.~Dye and L.~Kauffman~\cite{dye2009virtual}, extended bracket polynomial by L.~Kauffman~\cite{kauffman2009extended}, index polynomial by \linebreak Y.-H.~Im, K.~Lee and S.-Y.~Lee~\cite{im2010index} and zero polynomial by M.-J.~Jeong~\cite{jeong2016zero}. 

In this paper, our aim is to introduce new polynomial invariants for virtual knots. We define a sequence of polynomial invariants, $\{ L^n_K(t,\ell)\}_{n=1}^{\infty}$, which we call $L$-\emph{polynomials}, and a sequence of polynomial invariants, $\{ F^n_K(t,\ell)\}_{n=1}^{\infty}$, which we call $F$-\emph{polynomials}. The motivation for $L$-polynomials comes from Kauffman affine index polynomial $P_{K}(t)$~\cite{kauffman2013affine}. Recall that for an oriented virtual knot $K$ this polynomial is defined via its diagram $D$ by 
$$ 
P_{K}(t) = \sum_{c}\operatorname{sgn}(c)(t^{W_D(c)}-1), 
$$
where $\operatorname{sgn}(c)$ denotes the sign of the crossing $c$ in the oriented diagram $D$, and $W_D(c)$ is the weight, associated to the crossing $c$.  

Let $K$ be an oriented virtual knot and $D$ be its diagram. For a positive integer $n$, we consider $n$-th $L$-polynomial of $K$ by assigning two weights for each classical crossing  $c \in D$. One is the \emph{index value} $\operatorname{Ind}(c)$, which was defined in~\cite{cheng2013polynomial} and coincide with $W_D(c)$.  Second is the $n$-\emph{dwrithe} number $\nabla J _{n}(D)$, defined as difference between $n$-writhe and $(-n)$-writhe, with $n$-writhe defined in~\cite{satoh2014writhes}. For each classical crossing $c$ of diagram $D$ we smooth it locally to obtain a virtual knot diagram $D_c$ with one less classical crossing. The smoothing rule is  shown below in Fig.~\ref{fig5}. After smoothing, we calculate $n$-dwrithe value $\nabla J_{n}(D_{c})$ of $D_c$ and assign it to the crossing $c$ of $D$.  Then we define an \emph{$n$-th $L$-polynomial} of $K$ as 
$$
L_{K}^{n}(t,\ell) = \sum_{c\in C(D)}\operatorname{sgn}(c)(t^{\operatorname{Ind}(c)}\ell^{\abs{\nabla J_{n}(D_{c})}}-\ell^{\abs{\nabla J_{n}(D)}}), 
$$
where $C(D)$ denote the set of all classical crossings of $D$. Remark that  $L$-polynomials generalize the affine index polynomial since $P_K(t) = L^n_K (t, 1)$.  

We observe that $L$-polynomials are sometime fail to distinguish two virtual knots, due to the absolute values in powers of variable $\ell$, see definition of $L^n_K(t,\ell)$. To resolve this problem we modify $L$-polynomials and introduce $F$-polynomials. An \emph{$n$-th $F$-polynomial} of oriented virtual knot $K$ is defined via its diagram $D$ as
$$
\begin{gathered}
F_{K}^{n}(t,\ell) = \sum_{c \in C(D)} \operatorname{sgn}(c)t^{\text{Ind}(c)} \ell^{\nabla J_{n}(D_{c})}  \qquad \qquad \qquad  \\ 
\qquad \qquad \qquad \qquad -  \sum _{c\in T_{n}(D)} \operatorname{sgn}(c) \ell^{\nabla J_{n}(D_{c})} - \sum _{c\notin T_{n}(D)} \operatorname{sgn} (c) \ell^{\nabla J_{n}(D)}, 
\end{gathered}
$$ 
where  $T_n(D)$ is a set of crossings of $D$ with the following property: 
$$
T_{n}(D)=\{c \in D  \, \mid  \,  \nabla J_{n}(D_{c}) = \pm \nabla J_{n}(D)\}.
$$ 
These $F$-polynomials are more general than $L$-polynomials and distinguish many virtual knots, which can not be distinguished by $L$-polynomials. 

The paper is organized as follows. �In Section~\ref{affine sec} we recall definitions of affine index and affine index polynomial. Then we define $n$-dwrithe $\nabla J_n (D)$; prove that it is a flat virtual knot invariant (see Lemma~\ref{lemma2}); and describe how it change if we replace $D$ by its inverse $D^-$ or its mirror image $D^*$ (see Lemma~\ref{lemma1}). In Section~\ref{pol} for any positive integer $n$ we define $n$-th $L$-polynomial of oriented virtual knot  diagram and give an example of its computation for a reader convenience. After that we prove that any $n$-th  $L$-polynomial is a virtual knot invariant (see Theorem~\ref{thm-h}).  We observe that $L$-polynomials coincide with the affine index polynomial for classical knots (see Proposition~\ref{prop3.4}). At the end of the section we give an example (see Example~\ref{example3.6}) of oriented virtual knots for which the affine index polynomials and the writhe polynomials are trivial, but their $1$-st $L$-polynomials are non-trivial. In Section~\ref{inv}, we discuss the behavior of $L$-polynomials  under reflection and inversion  (see Theorem~\ref{th-mirror}). In Section~\ref{cs} we deal with the cosmetic crossing change conjecture for virtual knots. We prove that a crossing $c$ is not a cosmetic crossing if $\operatorname{Ind}(c)\neq0$ or $\nabla J_{n}(D_{c})\neq \pm \nabla J_n(D)$ for some $n$ (see Theorem~\ref{theorem5.2}). In Section~\ref{F-poly} for any positive integer $n$ we define $n$-th $F$-polynomial of oriented virtual knot diagram. We prove that for any $n$-th $F$-polynomial is a virtual knot invariant (see Theorem~\ref{gth}). The Example~\ref{ex6.5} gives a pair of oriented virtual knots which are distinguished by $F$-polynomials, whereas the writhe polynomial and $L$-polynomials fails to make distinction between these. 
In Section~\ref{mutbyp} we demonstrate in Examples~\ref{ex4.2} and~\ref{ex4.3} that $F$-polynomials are able to distinguish positive reflection mutants while the affine index polynomial fails to do it.

\section{Index value and dwrithe}
\label{affine sec}

Let $D$ be an oriented virtual knot diagram. By an \emph{arc} we mean an edge between two consecutive classical crossings along the orientation.  The sign of classical crossing $c \in C(D)$, denoted by $\operatorname{sgn}(c)$, is defined as in Fig.~\ref{fig2}.  
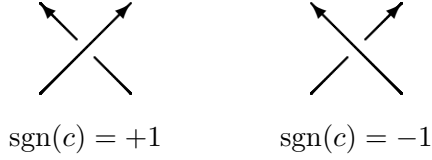
\begin{figure}[!ht]
\centering 
\unitlength=0.6mm
\begin{picture}(0,35)
\thicklines
\qbezier(-40,10)(-40,10)(-20,30)
\qbezier(-40,30)(-40,30)(-32,22) 
\qbezier(-20,10)(-20,10)(-28,18)
\put(-35,25){\vector(-1,1){5}}
\put(-25,25){\vector(1,1){5}}
\put(-30,0){\makebox(0,0)[cc]{$\operatorname{sgn}(c)=+1$}}
\qbezier(40,10)(40,10)(20,30)
\qbezier(40,30)(40,30)(32,22) 
\qbezier(20,10)(20,10)(28,18)
\put(25,25){\vector(-1,1){5}}
\put(35,25){\vector(1,1){5}}
\put(30,0){\makebox(0,0)[cc]{$\operatorname{sgn}(c)=-1$}}
\end{picture}
\caption{Crossing signs.} \label{fig2}
\end{figure}

Now assign an integer value to each arc in $D$ in such a way that the labeling around each crossing point of $D$ follows the rule as shown in Fig.~\ref{fig3}. L.~Kauffman proved in \cite[Proposition~4.1]{kauffman2013affine} that such integer labeling, called a \emph{Cheng coloring}, always exists for an oriented virtual knot diagram. Indeed, for an arc $\alpha$ of $D$ one can take label $\lambda (\alpha) = \sum_{c \in O(\alpha)} \operatorname{sgn} (c)$, where $O(\alpha)$ denotes the set of crossings first met as overcrossings on traveling along orientation, starting at the arc $\alpha$. 
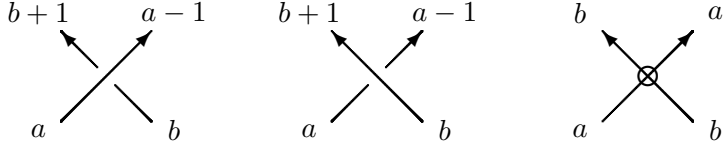
\begin{figure}[!ht]
\centering 
\unitlength=0.6mm
\begin{picture}(0,35)(0,5)
\thicklines
\qbezier(-70,10)(-70,10)(-50,30)
\qbezier(-70,30)(-70,30)(-62,22) 
\qbezier(-50,10)(-50,10)(-58,18)
\put(-65,25){\vector(-1,1){5}}
\put(-55,25){\vector(1,1){5}}
\put(-75,34){\makebox(0,0)[cc]{$b+1$}}
\put(-75,8){\makebox(0,0)[cc]{$a$}}
\put(-45,8){\makebox(0,0)[cc]{$b$}}
\put(-45,34){\makebox(0,0)[cc]{$a-1$}}
\qbezier(10,10)(10,10)(-10,30)
\qbezier(10,30)(10,30)(2,22) 
\qbezier(-10,10)(-10,10)(-2,18)
\put(-5,25){\vector(-1,1){5}}
\put(5,25){\vector(1,1){5}}
\put(-15,34){\makebox(0,0)[cc]{$b+1$}}
\put(-15,8){\makebox(0,0)[cc]{$a$}}
\put(15,8){\makebox(0,0)[cc]{$b$}}
\put(15,34){\makebox(0,0)[cc]{$a-1$}}
\qbezier(70,10)(70,10)(50,30)
\qbezier(70,30)(70,30)(50,10) 
\put(55,25){\vector(-1,1){5}}
\put(65,25){\vector(1,1){5}}
\put(60,20){\circle{4}}
\put(45,34){\makebox(0,0)[cc]{$b$}}
\put(45,8){\makebox(0,0)[cc]{$a$}}
\put(75,8){\makebox(0,0)[cc]{$b$}}
\put(75,34){\makebox(0,0)[cc]{$a$}}
\end{picture}
\caption{Labeling around crossing.} \label{fig3}
\end{figure}

 After labeling assign a weight $W_{D}(c)$ to each classical crossing $c$ is defined in  \cite{kauffman2013affine} as 
$$
W_{D}(c) = \operatorname{sgn} (c)(a-b-1). 
$$ 
Then the \emph{affine index polynomial} of virtual knot diagram $D$ is defined as
\begin{equation}
P_{D}(t) = \sum _{c \in C(D)} \operatorname{sgn}(c)(t^{W_{D}(c)}-1) \label{eqn1}
\end{equation} 
where the summation runs over the set $C(D)$ of classical crossings of $D$. In \cite{folwaczny2013linking}, L.~Folwaczny and L.~Kauffman  defined an invariant of virtual knots, called \emph{wriggle polynomial}, and proved that the wriggle polynomial is an alternate definition of the affine index polynomial.  

In \cite{cheng2013polynomial}, Z. Cheng and H. Gao assigned an integer value, called \emph{index value}, to each classical crossing $c$ of a virtual knot diagram and denoted it by $\operatorname{Ind}(c)$. It was proved  \cite[Theorem~3.6]{cheng2013polynomial} that  
\begin{equation}
\operatorname{Ind}(c) = W_{D}(c) = \operatorname{sgn} (c)(a-b-1)  \label{eq2.2}
\end{equation}
with $a$ and $b$ be labels as presented in Fig.~\ref{fig3}. 
Therefore, we can compute the index value through labeling procedure as  given in Fig.~\ref{fig3} and replace $W_{D}(c)$ by $\operatorname{Ind}(c)$ in equation~(\ref{eqn1}). Hence the affine index polynomial can be rewritten as   
$$
P_{D}(t) = \sum _{c\in C(D)} \operatorname{sgn}(c) (t^{\operatorname{Ind}(c)}-1) .
$$

In \cite{satoh2014writhes}, S.~Satoh and K.~Taniguchi introduced the $n$-th writhe. For each $n \in \mathbb{Z}\setminus \{0\}$ the  \emph{$n$-th writhe $J_n(D)$} of an oriented virtual link diagram $D$ is defined as the number of positive sign crossings minus number of negative sign crossings of $D$ with index value $n$. Remark, that $J_n (D)$ is indeed coefficient of $t^n$ in the affine index polynomial. This $n$-th writhe is a virtual knot invariant, for more details we refer to \cite{satoh2014writhes}. Using $n$-th writhe, we define a new invariant as follows. 

 \begin{definition} {\rm 
 Let $n\in \mathbb{N}$ and $D$ be an oriented virtual knot diagram. Then the \emph{$n$-th dwrithe} of $D$, denoted by $\nabla J_{n}(D)$, is defined as   
 $$
 \nabla J_{n}(D)=J_{n}(D)-J_{-n}(D).
 $$
 }
 \end{definition}
 
\begin{remark} \label{rem-inv} 
{\rm 
The $n$-th dwrithe  $\nabla J_{n}(D)$ is a virtual knot invariant, since  $n$-th writhe $J_n(D)$ is an oriented virtual knot invariant by~\cite{satoh2014writhes}. 
}
\end{remark} 

Obviously, $\nabla J_{n}(D) =0$ for any classical knot diagram. 
  
 \begin{remark} \label{rem2.2} {\rm 
Let $D$ be a virtual knot diagram. Consider set of all affine index values  of crossing points:  
$$
S(D)=\{ \abs{ \operatorname{Ind}(c) } \, : \, c \in C(D)\} \subset \mathbb{N}. 
$$ 
Then $\nabla J_n(D)=0$ for any $n\in \mathbb{N}\setminus S(D)$. 
}
\end{remark}

A \emph{flat virtual knot diagram} is a virtual knot diagram obtained by forgetting the over/under-information of every real crossing.  It means that a flat virtual knot is an equivalence class of flat virtual knot diagrams by \emph{flat Reidemeister moves} which are Reidemeister moves (see Figs.~\ref{fig1a} and \ref{fig1b}) without the over/under information. We will say that a virtual knot invariant is a \emph{flat virtual knot invariant}, if it is independent of crossing change operation.  

\begin{lemma} \label{lemma2}
For any $n \in \mathbb{N}$, the $n$-th dwrithe $\nabla J_n (D)$ is a flat virtual knot invariant. 
\end{lemma} 

\begin{proof} 
As we observed above, for any $n \in \mathbb{N}$ dwrithe $\nabla J_{n}(D)$ is a virtual knot invariant. To prove that it is a flat virtual knot invariant, we need to show that $\nabla J_{n}(D)$ is an invariant under the crossing change operation. 

Let $D'$ be the diagram obtained from $D$ by applying crossing change operation at a crossing $c$ and let $c'$ be the corresponding crossing in $D'$. Then $\operatorname{Ind}(c')=-\operatorname{Ind}(c)$ and $\operatorname{sgn}(c')=-\operatorname{sgn}(c)$. If $n \neq \pm \operatorname{Ind}(c) $, then   $J_{n}(D')=J_{n}(D)$ and  $J_{-n}(D')=J_{-n}(D)$. If $n = \operatorname{Ind}(c)$ then  $J_{n}(D') = J_{n}(D) - \operatorname{sgn}(c)$ and  $J_{-n}(D') = J_{-n} (D) + \operatorname{sgn} (c') = J_{-n}(D) - \operatorname{sgn}(c)$, hence $\nabla J_{n}(D') = \nabla J_{n}(D)$. Analogously, if $n = -  \operatorname{Ind}(c)$ then  $J_{-n}(D') = J_{-n}(D) - \operatorname{sgn}(c)$ and  $J_{n}(D') = J_{n} (D) + \operatorname{sgn} (c') = J_{n}(D) - \operatorname{sgn}(c)$, hence $\nabla J_{n}(D') = \nabla J_{n}(D)$
Thus, $n$-th dwrithe is invariant under crossing change operations and it is a flat virtual knot invariant.
\end{proof} 
     
Let $D^-$ be the \emph{reverse} of $D$, obtained from $D$ by reversing the orientation and let $D^*$ be the \emph{mirror image} of $D$, obtained by switching all the classical crossings in~$D$. 

\begin{lemma} \label{lemma1}
If $D$ is an oriented virtual knot diagram, then $\nabla J_{n}(D^*)=\nabla J_{n}(D)$ and $\nabla J_{n}(D^-)=-\nabla J_{n}(D)$. 
\end{lemma} 

\begin{proof}  Let $c$ be a crossing in $D$, and $c^*$ and $c^-$ be the corresponding crossings in $D^*$ and $D^-$, respectively.  The case when $\operatorname{sgn}(c)=1$ is presented in Fig.~\ref{fig4new}. 
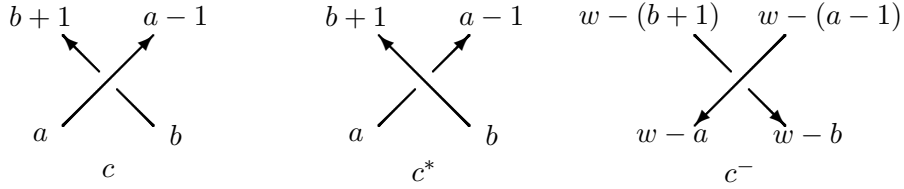
\begin{figure}[!ht]
\centering 
\unitlength=0.6mm
\begin{picture}(0,35)
\thicklines
\qbezier(-80,10)(-80,10)(-60,30)
\qbezier(-80,30)(-80,30)(-72,22) 
\qbezier(-60,10)(-60,10)(-68,18)
\put(-75,25){\vector(-1,1){5}}
\put(-65,25){\vector(1,1){5}}
\put(-85,34){\makebox(0,0)[cc]{$b+1$}}
\put(-85,8){\makebox(0,0)[cc]{$a$}}
\put(-55,8){\makebox(0,0)[cc]{$b$}}
\put(-55,34){\makebox(0,0)[cc]{$a-1$}}
\put(-70,0){\makebox(0,0)[cc]{$c$}}
\qbezier(10,10)(10,10)(-10,30)
\qbezier(10,30)(10,30)(2,22) 
\qbezier(-10,10)(-10,10)(-2,18)
\put(-5,25){\vector(-1,1){5}}
\put(5,25){\vector(1,1){5}}
\put(-15,34){\makebox(0,0)[cc]{$b+1$}}
\put(-15,8){\makebox(0,0)[cc]{$a$}}
\put(15,8){\makebox(0,0)[cc]{$b$}}
\put(15,34){\makebox(0,0)[cc]{$a-1$}}
\put(0,0){\makebox(0,0)[cc]{$c^{*}$}}
\qbezier(80,10)(80,10)(72,18)
\qbezier(60,30)(60,30)(68,22)
\qbezier(80,30)(80,30)(60,10) 
\put(65,15){\vector(-1,-1){5}}
\put(75,15){\vector(1,-1){5}}
\put(50,34){\makebox(0,0)[cc]{$w - (b+1)$}}
\put(55,8){\makebox(0,0)[cc]{$w-a$}}
\put(85,8){\makebox(0,0)[cc]{$w-b$}}
\put(90,34){\makebox(0,0)[cc]{$w-(a-1)$}}
\put(70,0){\makebox(0,0)[cc]{$c^-$}}
\end{picture}
\caption{Labeling around crossings.} \label{fig4new}
\end{figure}

It is clear from Fig.~\ref{fig4new}, that $\operatorname{sgn}(c^-) = \operatorname{sgn} (c)$ and $\operatorname{sgn} (c^*) = - \operatorname{sgn}(c)$. It is easy to see from the definition (see Fig.~\ref{fig3}), that Cheng coloring of $D^*$ coincides with Cheng coloring of $D$. Hence $\operatorname{Ind} (c^*) = - \operatorname{Ind}(c)$. To obtain Cheng coloring of $D^-$ we refer to \cite[Proposition~4.2]{kauffman2013affine}: if $\lambda (\alpha)$ is the above defined labeling function that count overcrossings with signs, then for the arc $\alpha^{-} \in D^{-}$, corresponding to arc $\alpha \in D$, the following property holds: $\lambda(\alpha) + \lambda(\alpha^{-}) = w$, where $w = \operatorname{wr} (D)$ is the writhe number of $D$. Hence $\operatorname{Ind} (c^-) = - \operatorname{Ind} (c)$. 

Consider the set $S(D)$ introduced for a diagram $D$ in Remark~\ref{rem2.2}. Assume that  $n \notin S(D)$. Then by Remark~\ref{rem2.2} we have $\nabla J_{n}(D) = 0$, and hence $\nabla J_{n}(D^*) = \nabla J_{n}(D)=0$ and $\nabla J_{n}(D^-) = \nabla J_{n}(D) =0$.  

Now assume that $n \in S(D)$. Then we have   
$$
J_{\pm n}(D^*) = -J_{\mp n}(D) \quad \text{and} \quad  J_{\pm n}(D^-)=J_{\mp n}(D).
$$ 
Therefore, 
$$
\nabla J_{n}(D^*) = J_{n}(D^{*}) - J_{-n}(D^*) =- J_{-n}(D)-(-J_{n}(D))=\nabla J_{n}(D), 
$$
and, analogously,  
$$
\nabla J_{n}(D^-) = J_{n}(D^-) - J_{-n}(D^-) = J_{-n}(D)-J_{n}(D) = -\nabla J_{n}(D). 
$$ 
\end{proof}

\section{L-polynomials of virtual knot diagrams}  \label{pol}

Let $c$ be a classical crossing  of an oriented virtual knot diagram $D$. There are two possibility to smooth in $c$. One is to smooth \emph{along} the orientation of arcs shown in~Fig.~\ref{fig4}. 
\begin{figure}[!ht]
\centering 
\unitlength=0.6mm
\begin{picture}(0,25)(0,10)
\thicklines
\qbezier(-70,10)(-70,10)(-50,30)
\qbezier(-70,30)(-70,30)(-62,22) 
\qbezier(-50,10)(-50,10)(-58,18)
\put(-65,25){\vector(-1,1){5}}
\put(-55,25){\vector(1,1){5}}
\qbezier(10,10)(10,10)(-10,30)
\qbezier(10,30)(10,30)(2,22) 
\qbezier(-10,10)(-10,10)(-2,18)
\put(-5,25){\vector(-1,1){5}}
\put(5,25){\vector(1,1){5}}
\put(-30,20){\makebox(0,0)[cc]{or}}
\put(30,20){\makebox(0,0)[cc]{$\longrightarrow$}}
\qbezier(50,10)(60,20)(50,30)
\qbezier(70,10)(60,20)(70,30) 
\put(55,25){\vector(-1,1){5}}
\put(65,25){\vector(1,1){5}}
\end{picture}
\caption{Smoothing along orientation.} \label{fig4}
\end{figure}
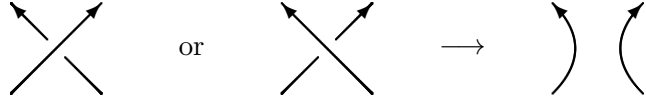

Another is smoothing \emph{against}  the orientation of arcs shown in Fig.~\ref{fig5}. 
\begin{figure}[!ht]
\centering 
\unitlength=0.6mm
\begin{picture}(0,25)(0,10)
\thicklines
\qbezier(-80,10)(-80,10)(-60,30)
\qbezier(-80,30)(-80,30)(-72,22) 
\qbezier(-60,10)(-60,10)(-68,18)
\put(-75,25){\vector(-1,1){5}}
\put(-65,25){\vector(1,1){5}}
\put(-50,20){\makebox(0,0)[cc]{$\longrightarrow$}}
\qbezier(-40,30)(-30,20)(-20,30)
\qbezier(-40,10)(-30,20)(-20,10) 
\put(-35,15){\vector(-1,-1){5}}
\put(-25,25){\vector(1,1){5}}
\put(0,20){\makebox(0,0)[cc]{and}}
\qbezier(20,30)(20,30)(40,10)
\qbezier(20,10)(20,10)(28,18) 
\qbezier(40,30)(40,30)(32,22)
\put(25,25){\vector(-1,1){5}}
\put(35,25){\vector(1,1){5}}
\put(50,20){\makebox(0,0)[cc]{$\longrightarrow$}}
\qbezier(60,30)(70,20)(80,30)
\qbezier(60,10)(70,20)(80,10) 
\put(65,25){\vector(-1,1){5}}
\put(75,15){\vector(1,-1){5}}
\end{picture}
\caption{Smoothing against orientation.} \label{fig5}
\end{figure}
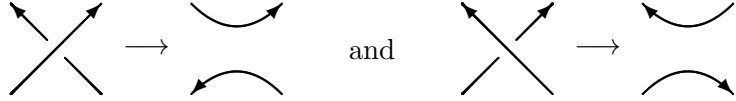

Let us denote by $D_{c}$ the oriented diagram obtained from $D$ by smoothing at $c$ against the orientation of arcs. The orientation of $D_{c}$ is induced by the orientation of smoothing.  Since $D$ is a virtual knot diagram, $D_{c}$ is also a virtual knot diagram. 

\begin{definition} \label{h} {\rm 
Let $D$ be an oriented virtual knot diagram and $n$ be any positive integer. Then \emph{$n$-th $L$-polynomial} of $D$ at $n$ is define as 
$$
L_{D}^{n}(t,\ell) = \sum_{c \in C(D)} \operatorname{sgn}(c)(t^{\operatorname{Ind}(c)} \ell^{\abs{\nabla J_{n}(D_{c})}} -\ell^{\abs{\nabla J_{n}(D)}}).
$$
}
\end{definition}

Since number of crossings in a virtual knot diagram is finite, it is easy to see that for a given virtual knot diagram $D$ there exists a positive integer $N$, such that $\nabla J_{n}(D)=0$ and $\nabla J_{n}(D_{c})=0$ for all $n > N$ and $c \in C(D)$. Therefore, $L^{n}_{D}(t,\ell) = P_{D}(t)$ for all $n >  N$. 

More precisely, let us denote by $\abs{C(D)}$ the cardinality of the set of classical crossings of $D$. Then by \cite[Proposition~4.1]{kauffman2013affine} for any arc of $D$ the absolute value of  its label is at most $\abs{C(D)}$.  Remark that  $\abs{C(D_c)} =\abs{C(D)} -1$ for any $c \in C(D)$. Hence, for any crossing $c'$ in $D$ or  $D_c$ the inequality $| \operatorname{Ind} (c')  | \leq 2 \abs{C(D)} + 1$ holds. Thus for any $n > 2  \abs{C(D)} + 1$, we have  $\nabla J_{n}(D_{c})=\nabla J_{n}(D)=0$ and $L^{n}_{D}(t,\ell)=P_{D}(t)$.

For a diagram $D$ we consider  a set $\mathcal{N}_D \subset \mathbb{N}$ defined by  
$$
\mathcal{N}_D=\displaystyle\bigcup_{c\in C(D)} S_{D_c}\cup S_{D},  
$$ 
where $S(D)$ and $S(D_c)$ are as in Remark~\ref{rem2.2}. 
Then $L_{D}^{n}(t,\ell)=P_{D}(t)$ for any $n \notin \mathcal{N}_D$. 

Before discussing properties of $L$-polynomial we give an example of its calculation. 

\begin{example}  \label{ex3.2} {\rm 
Let us consider an oriented virtual knot diagram $D$ presented in~Fig.~\ref{fig6}. 
\begin{figure}[!ht]
\centering 
\unitlength=0.34mm
\begin{picture}(0,165)
\put(-60,0){\begin{picture}(0,160)
\thicklines
\qbezier(10,160)(10,160)(-10,140) 
\qbezier(-10,160)(-10,160)(-3,153)
\qbezier(10,140)(10,140)(3,147)
\qbezier(10,140)(10,140)(-10,120) 
\qbezier(-10,140)(-10,140)(-3,133)
\qbezier(10,120)(10,120)(3,127)
\qbezier(-10,120)(-10,120)(-10,100)
\qbezier(10,120)(10,120)(30,100)
\qbezier(10,160)(30,160)(30,140)
\qbezier(30,140)(30,140)(30,120)
\qbezier(30,120)(30,120)(10,100)
\put(20,110){\circle{6}}
\put(30,140){\vector(0,1){5}}
\qbezier(-10,100)(-10,100)(10,80)
\qbezier(10,100)(10,100)(-10,80)
\put(0,90){\circle{6}}
\qbezier(-10,80)(-10,80)(10,60)
\qbezier(10,80)(10,80)(3,73) 
\qbezier(-10,60)(-10,60)(-3,67) 
\qbezier(30,100)(30,100)(30,60)
\qbezier(-10,60)(-10,60)(-10,20)
\qbezier(10,60)(10,60)(30,40)
\qbezier(10,40)(10,40)(30,60)
\put(20,50){\circle{6}}
\qbezier(10,40)(10,40)(30,20)
\qbezier(30,40)(30,40)(23,33)
\qbezier(10,20)(10,20)(17,27)
\qbezier(-10,20)(-10,20)(10,0)
\qbezier(-10,0)(-10,0)(10,20)
\put(0,10){\circle{6}}
\qbezier(10,0)(30,0)(30,20)
\qbezier(-10,0)(-30,0)(-30,20)
\qbezier(-30,20)(-30,20)(-30,140)
\qbezier(-30,140)(-30,160)(-10,160)
\put(-15,150){\makebox(0,0)[cc]{\footnotesize $\alpha$}}
\put(-15,130){\makebox(0,0)[cc]{\footnotesize $\beta$}}
\put(-15,70){\makebox(0,0)[cc]{\footnotesize $\gamma$}}
\put(35,30){\makebox(0,0)[cc]{\footnotesize $\delta$}}
\end{picture}} 
%%%
\put(60,0){\begin{picture}(0,160)
\thicklines
\qbezier(10,160)(10,160)(-10,140) 
\qbezier(-10,160)(-10,160)(-3,153)
\qbezier(10,140)(10,140)(3,147)
\qbezier(10,140)(10,140)(-10,120) 
\qbezier(-10,140)(-10,140)(-3,133)
\qbezier(10,120)(10,120)(3,127)
\qbezier(-10,120)(-10,120)(-10,100)
\qbezier(10,120)(10,120)(30,100)
\qbezier(10,160)(30,160)(30,140)
\qbezier(30,140)(30,140)(30,120)
\qbezier(30,120)(30,120)(10,100)
\put(20,110){\circle{6}}
\qbezier(-10,100)(-10,100)(10,80)
\qbezier(10,100)(10,100)(-10,80)
\put(0,90){\circle{6}}
\qbezier(-10,80)(-10,80)(10,60)
\qbezier(10,80)(10,80)(3,73) 
\qbezier(-10,60)(-10,60)(-3,67) 
\qbezier(30,100)(30,100)(30,60)
\qbezier(-10,60)(-10,60)(-10,20)
\qbezier(10,60)(10,60)(30,40)
\qbezier(10,40)(10,40)(30,60)
\put(20,50){\circle{6}}
\qbezier(10,40)(10,40)(30,20)
\qbezier(30,40)(30,40)(23,33)
\qbezier(10,20)(10,20)(17,27)
\qbezier(-10,20)(-10,20)(10,0)
\qbezier(-10,0)(-10,0)(10,20)
\put(0,10){\circle{6}}
\qbezier(10,0)(30,0)(30,20)
\qbezier(-10,0)(-30,0)(-30,20)
\qbezier(-30,20)(-30,20)(-30,140)
\qbezier(-30,140)(-30,160)(-10,160)
\put(0,150){\vector(-1,-1){8}}
\put(-3,153){\vector(-1,1){5}}
\put(0,130){\vector(1,1){8}}
\put(3,127){\vector(1,-1){5}}
\put(3,73){\vector(1,1){5}}
\put(0,70){\vector(-1,1){8}}
\put(23,33){\vector(1,1){5}}
\put(20,30){\vector(1,-1){8}}
\put(-20,165){\makebox(0,0)[cc]{\footnotesize $-2$}}
\put(-15,140){\makebox(0,0)[cc]{\footnotesize $1$}}
\put(18,140){\makebox(0,0)[cc]{\footnotesize $-1$}}
\put(15,123){\makebox(0,0)[cc]{\footnotesize $0$}}
\put(35,133){\makebox(0,0)[cc]{\footnotesize $0$}}
\put(-18,110){\makebox(0,0)[cc]{\footnotesize $-2$}}
\put(15,97){\makebox(0,0)[cc]{\footnotesize $0$}}
\put(35,80){\makebox(0,0)[cc]{\footnotesize $0$}}
\put(-15,80){\makebox(0,0)[cc]{\footnotesize $0$}}
\put(18,80){\makebox(0,0)[cc]{\footnotesize $-2$}}
\put(15,63){\makebox(0,0)[cc]{\footnotesize $-1$}}
\put(-20,40){\makebox(0,0)[cc]{\footnotesize $-1$}}
\put(5,40){\makebox(0,0)[cc]{\footnotesize $0$}}
\put(40,40){\makebox(0,0)[cc]{\footnotesize $-1$}}
\put(0,20){\makebox(0,0)[cc]{\footnotesize $-2$}}
\put(40,20){\makebox(0,0)[cc]{\footnotesize $-1$}}
\end{picture}} 
\end{picture}
\caption{Oriented virtual diagram $D$ and its labeling.} \label{fig6}
\end{figure}
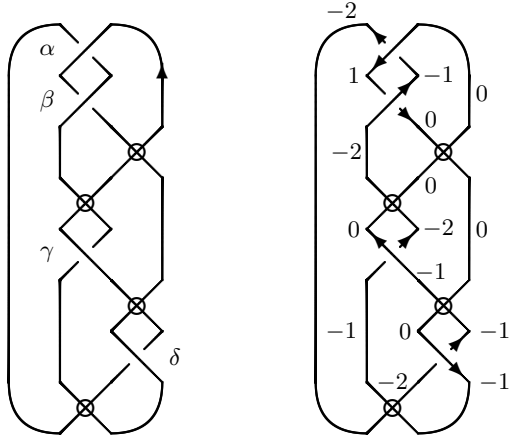
The diagram $D$ has four classical crossings denoted by $\alpha$, $\beta$, $\gamma$, and $\delta$, see the left-hand picture. In the right-hand picture we presented orientation of arcs for each classical crossing and the corresponding labeling, satisfying the rule given in Fig.~\ref{fig2}. 
Crossing signs can be easy found from arc orientations around crossing points given in Fig.~\ref{fig6}: 
$$
\operatorname{sgn}(\alpha) = \operatorname{sgn}(\beta) = \operatorname{sgn}(\gamma) =-1 \qquad  \text{and} \qquad \operatorname{sgn}(\delta)=1.
$$
Index values can be calculated directly from crossing signs and labeling of arcs by Eq.~(\ref{eq2.2}): 
$$
\operatorname{Ind}(\alpha)=2,  \quad \operatorname{Ind}(\beta)=-2, \quad \operatorname{Ind}(\gamma)=1, \quad \operatorname{Ind}(\delta)=1. 
$$
Therefore, only the following writhe numbers can be non-trivial: $J_1(D)$, $J_2(D)$, and $J_{-2}(D)$. It is easy to see, that $J_1(D) = \operatorname{sgn} (\gamma) + \operatorname{sgn}(\delta) = 0$, $J_2 (D) = \operatorname{sgn}(\alpha) = -1$, and $J_{-2} (D) = \operatorname{sgn}(\beta) = -1$. Then $\nabla J_1(D) = J_1(D) - J_{-1}(D) = 0$ and $\nabla J_2 (D) = J_2(D) - J_{-2}(D) =0$. For any $n \geq 3$ we have $\nabla J_n (D) =0$.  

Let us consider against orientation smoothings at classical crossing points $\alpha$, $\beta$, $\gamma$, and $\delta$ of $D$. The resulting oriented virtual knot diagrams $D_{\alpha}$, $D_{\beta}$, $D_{\gamma}$, and $D_{\delta}$ are presented in Fig~\ref{fig7}, where classical crossing points have the same notations as in $D$ and labelings of arcs (Cheng colorings) are given.  Orientations of these diagrams are induced by orientations of smoothings. 
\begin{figure}[!ht]
\centering 
\unitlength=0.34mm
\begin{picture}(0,180)(0,-10)
%%%
\put(-150,0){\begin{picture}(0,170)
\thicklines
\qbezier(10,160)(10,160)(10,140) 
\qbezier(-10,160)(-10,160)(-10,140)
\qbezier(10,140)(10,140)(-10,120) 
\qbezier(-10,140)(-10,140)(-3,133)
\qbezier(10,120)(10,120)(3,127)
\qbezier(-10,120)(-10,120)(-10,100)
\qbezier(10,120)(10,120)(30,100)
\qbezier(10,160)(30,160)(30,140)
\qbezier(30,140)(30,140)(30,120)
\qbezier(30,120)(30,120)(10,100)
\put(20,110){\circle{6}}
\qbezier(-10,100)(-10,100)(10,80)
\qbezier(10,100)(10,100)(-10,80)
\put(0,90){\circle{6}}
\qbezier(-10,80)(-10,80)(10,60)
\qbezier(10,80)(10,80)(3,73) 
\qbezier(-10,60)(-10,60)(-3,67) 
\qbezier(30,100)(30,100)(30,60)
\qbezier(-10,60)(-10,60)(-10,20)
\qbezier(10,60)(10,60)(30,40)
\qbezier(10,40)(10,40)(30,60)
\put(20,50){\circle{6}}
\qbezier(10,40)(10,40)(30,20)
\qbezier(30,40)(30,40)(23,33)
\qbezier(10,20)(10,20)(17,27)
\qbezier(-10,20)(-10,20)(10,0)
\qbezier(-10,0)(-10,0)(10,20)
\put(0,10){\circle{6}}
\qbezier(10,0)(30,0)(30,20)
\qbezier(-10,0)(-30,0)(-30,20)
\qbezier(-30,20)(-30,20)(-30,140)
\qbezier(-30,140)(-30,160)(-10,160)
\put(-15,130){\makebox(0,0)[cc]{\footnotesize $\beta$}}
\put(-15,70){\makebox(0,0)[cc]{\footnotesize $\gamma$}}
\put(35,30){\makebox(0,0)[cc]{\footnotesize $\delta$}}
\put(-10,150){\vector(0,-1){5}}
\put(10,150){\vector(0,1){5}}
\put(0,130){\vector(1,1){8}}
\put(3,127){\vector(1,-1){5}}
\put(3,73){\vector(1,1){5}}
\put(0,70){\vector(1,-1){8}}
\put(17,27){\vector(-1,-1){5}}
\put(20,30){\vector(1,-1){8}}
\put(-20,165){\makebox(0,0)[cc]{\footnotesize $-2$}}
\put(15,123){\makebox(0,0)[cc]{\footnotesize $-3$}}
\put(35,133){\makebox(0,0)[cc]{\footnotesize $0$}}
\put(-18,110){\makebox(0,0)[cc]{\footnotesize $-1$}}
\put(15,97){\makebox(0,0)[cc]{\footnotesize $0$}}
\put(35,80){\makebox(0,0)[cc]{\footnotesize $-3$}}
\put(-15,80){\makebox(0,0)[cc]{\footnotesize $0$}}
\put(18,80){\makebox(0,0)[cc]{\footnotesize $-1$}}
\put(15,63){\makebox(0,0)[cc]{\footnotesize $-1$}}
\put(-20,40){\makebox(0,0)[cc]{\footnotesize $-2$}}
\put(5,40){\makebox(0,0)[cc]{\footnotesize $-3$}}
\put(40,40){\makebox(0,0)[cc]{\footnotesize $-1$}}
\put(0,20){\makebox(0,0)[cc]{\footnotesize $-2$}}
\put(40,20){\makebox(0,0)[cc]{\footnotesize $-2$}}
\put(0,-10){\makebox(0,0)[cc]{$D_{\alpha}$}}
\end{picture}} 
%%%
\put(-50,0){\begin{picture}(0,160)
\thicklines
\qbezier(10,160)(10,160)(-10,140) 
\qbezier(-10,160)(-10,160)(-3,153)
\qbezier(10,140)(10,140)(3,147)
\qbezier(10,140)(10,140)(10,120) 
\qbezier(-10,140)(-10,140)(-10,120)
\qbezier(-10,120)(-10,120)(-10,100)
\qbezier(10,120)(10,120)(30,100)
\qbezier(10,160)(30,160)(30,140)
\qbezier(30,140)(30,140)(30,120)
\qbezier(30,120)(30,120)(10,100)
\put(20,110){\circle{6}}
\qbezier(-10,100)(-10,100)(10,80)
\qbezier(10,100)(10,100)(-10,80)
\put(0,90){\circle{6}}
\qbezier(-10,80)(-10,80)(10,60)
\qbezier(10,80)(10,80)(3,73) 
\qbezier(-10,60)(-10,60)(-3,67) 
\qbezier(30,100)(30,100)(30,60)
\qbezier(-10,60)(-10,60)(-10,20)
\qbezier(10,60)(10,60)(30,40)
\qbezier(10,40)(10,40)(30,60)
\put(20,50){\circle{6}}
\qbezier(10,40)(10,40)(30,20)
\qbezier(30,40)(30,40)(23,33)
\qbezier(10,20)(10,20)(17,27)
\qbezier(-10,20)(-10,20)(10,0)
\qbezier(-10,0)(-10,0)(10,20)
\put(0,10){\circle{6}}
\qbezier(10,0)(30,0)(30,20)
\qbezier(-10,0)(-30,0)(-30,20)
\qbezier(-30,20)(-30,20)(-30,140)
\qbezier(-30,140)(-30,160)(-10,160)
\put(-15,150){\makebox(0,0)[cc]{\footnotesize $\alpha$}}
\put(-15,70){\makebox(0,0)[cc]{\footnotesize $\gamma$}}
\put(35,30){\makebox(0,0)[cc]{\footnotesize $\delta$}}
\put(0,150){\vector(-1,-1){8}}
\put(-3,153){\vector(-1,1){5}}
\put(-10,130){\vector(0,-1){5}}
\put(10,130){\vector(0,1){5}}
\put(-3,67){\vector(-1,-1){5}}
\put(0,70){\vector(-1,1){8}}
\put(23,33){\vector(1,1){5}}
\put(20,30){\vector(-1,1){8}}
\put(-20,165){\makebox(0,0)[cc]{\footnotesize $2$}}
\put(-15,140){\makebox(0,0)[cc]{\footnotesize $1$}}
\put(18,140){\makebox(0,0)[cc]{\footnotesize $3$}}
\put(35,133){\makebox(0,0)[cc]{\footnotesize $0$}}
\put(15,97){\makebox(0,0)[cc]{\footnotesize $0$}}
\put(35,80){\makebox(0,0)[cc]{\footnotesize $3$}}
\put(-15,80){\makebox(0,0)[cc]{\footnotesize $0$}}
\put(18,80){\makebox(0,0)[cc]{\footnotesize $1$}}
\put(15,63){\makebox(0,0)[cc]{\footnotesize $1$}}
\put(-20,40){\makebox(0,0)[cc]{\footnotesize $2$}}
\put(5,40){\makebox(0,0)[cc]{\footnotesize $3$}}
\put(35,40){\makebox(0,0)[cc]{\footnotesize $1$}}
\put(0,20){\makebox(0,0)[cc]{\footnotesize $2$}}
\put(35,20){\makebox(0,0)[cc]{\footnotesize $2$}}
\put(0,-10){\makebox(0,0)[cc]{$D_{\beta}$}}
\end{picture}} 
%%%
\put(50,0){\begin{picture}(0,160)
\thicklines
\qbezier(10,160)(10,160)(-10,140) 
\qbezier(-10,160)(-10,160)(-3,153)
\qbezier(10,140)(10,140)(3,147)
\qbezier(10,140)(10,140)(-10,120) 
\qbezier(-10,140)(-10,140)(-3,133)
\qbezier(10,120)(10,120)(3,127)
\qbezier(-10,120)(-10,120)(-10,100)
\qbezier(10,120)(10,120)(30,100)
\qbezier(10,160)(30,160)(30,140)
\qbezier(30,140)(30,140)(30,120)
\qbezier(30,120)(30,120)(10,100)
\put(20,110){\circle{6}}
\qbezier(-10,100)(-10,100)(10,80)
\qbezier(10,100)(10,100)(-10,80)
\put(0,90){\circle{6}}
\qbezier(-10,80)(0,70)(10,80)
\qbezier(-10,60)(0,70)(10,60) 
\qbezier(30,100)(30,100)(30,60)
\qbezier(-10,60)(-10,60)(-10,20)
\qbezier(10,60)(10,60)(30,40)
\qbezier(10,40)(10,40)(30,60)
\put(20,50){\circle{6}}
\qbezier(10,40)(10,40)(30,20)
\qbezier(30,40)(30,40)(23,33)
\qbezier(10,20)(10,20)(17,27)
\qbezier(-10,20)(-10,20)(10,0)
\qbezier(-10,0)(-10,0)(10,20)
\put(0,10){\circle{6}}
\qbezier(10,0)(30,0)(30,20)
\qbezier(-10,0)(-30,0)(-30,20)
\qbezier(-30,20)(-30,20)(-30,140)
\qbezier(-30,140)(-30,160)(-10,160)
\put(-15,150){\makebox(0,0)[cc]{\footnotesize $\alpha$}}
\put(-15,130){\makebox(0,0)[cc]{\footnotesize $\beta$}}
\put(35,30){\makebox(0,0)[cc]{\footnotesize $\delta$}}
\put(0,150){\vector(-1,-1){8}}
\put(3,147){\vector(1,-1){5}}
\put(0,130){\vector(-1,-1){8}}
\put(3,127){\vector(1,-1){5}}
\put(0,75){\vector(-1,0){5}}
\put(0,65){\vector(1,0){5}}
\put(17,27){\vector(-1,-1){5}}
\put(20,30){\vector(1,-1){8}}
\put(-20,165){\makebox(0,0)[cc]{\footnotesize $0$}}
\put(-20,140){\makebox(0,0)[cc]{\footnotesize $-1$}}
\put(18,140){\makebox(0,0)[cc]{\footnotesize $1$}}
\put(15,123){\makebox(0,0)[cc]{\footnotesize $0$}}
\put(35,133){\makebox(0,0)[cc]{\footnotesize $0$}}
\put(-18,110){\makebox(0,0)[cc]{\footnotesize $0$}}
\put(15,97){\makebox(0,0)[cc]{\footnotesize $0$}}
\put(35,80){\makebox(0,0)[cc]{\footnotesize $0$}}
\put(-15,80){\makebox(0,0)[cc]{\footnotesize $0$}}
\put(18,80){\makebox(0,0)[cc]{\footnotesize $0$}}
\put(15,63){\makebox(0,0)[cc]{\footnotesize $1$}}
\put(-20,40){\makebox(0,0)[cc]{\footnotesize $1$}}
\put(5,40){\makebox(0,0)[cc]{\footnotesize $0$}}
\put(35,40){\makebox(0,0)[cc]{\footnotesize $1$}}
\put(5,20){\makebox(0,0)[cc]{\footnotesize $0$}}
\put(35,20){\makebox(0,0)[cc]{\footnotesize $1$}}
\put(0,-10){\makebox(0,0)[cc]{$D_{\gamma}$}}
\end{picture}} 
%%%
\put(150,0){\begin{picture}(0,160)
\thicklines
\qbezier(10,160)(10,160)(-10,140) 
\qbezier(-10,160)(-10,160)(-3,153)
\qbezier(10,140)(10,140)(3,147)
\qbezier(10,140)(10,140)(-10,120) 
\qbezier(-10,140)(-10,140)(-3,133)
\qbezier(10,120)(10,120)(3,127)
\qbezier(-10,120)(-10,120)(-10,100)
\qbezier(10,120)(10,120)(30,100)
\qbezier(10,160)(30,160)(30,140)
\qbezier(30,140)(30,140)(30,120)
\qbezier(30,120)(30,120)(10,100)
\put(20,110){\circle{6}}
\qbezier(-10,100)(-10,100)(10,80)
\qbezier(10,100)(10,100)(-10,80)
\put(0,90){\circle{6}}
\qbezier(-10,80)(-10,80)(10,60)
\qbezier(10,80)(10,80)(3,73) 
\qbezier(-10,60)(-10,60)(-3,67) 
\qbezier(30,100)(30,100)(30,60)
\qbezier(-10,60)(-10,60)(-10,20)
\qbezier(10,60)(10,60)(30,40)
\qbezier(10,40)(10,40)(30,60)
\put(20,50){\circle{6}}
\qbezier(10,40)(10,40)(10,20)
\qbezier(30,40)(30,40)(30,20)
\qbezier(-10,20)(-10,20)(10,0)
\qbezier(-10,0)(-10,0)(10,20)
\put(0,10){\circle{6}}
\qbezier(10,0)(30,0)(30,20)
\qbezier(-10,0)(-30,0)(-30,20)
\qbezier(-30,20)(-30,20)(-30,140)
\qbezier(-30,140)(-30,160)(-10,160)
\put(-15,150){\makebox(0,0)[cc]{\footnotesize $\alpha$}}
\put(-15,130){\makebox(0,0)[cc]{\footnotesize $\beta$}}
\put(-15,70){\makebox(0,0)[cc]{\footnotesize $\gamma$}}
\put(0,150){\vector(1,1){8}}
\put(-3,153){\vector(-1,1){5}}
\put(0,130){\vector(1,1){8}}
\put(-3,133){\vector(-1,1){5}}
\put(3,73){\vector(1,1){5}}
\put(0,70){\vector(1,-1){8}}
\put(10,30){\vector(0,1){5}}
\put(30,30){\vector(0,-1){5}}
\put(-20,165){\makebox(0,0)[cc]{\footnotesize $0$}}
\put(-15,140){\makebox(0,0)[cc]{\footnotesize $1$}}
\put(18,140){\makebox(0,0)[cc]{\footnotesize $-1$}}
\put(15,123){\makebox(0,0)[cc]{\footnotesize $0$}}
\put(35,133){\makebox(0,0)[cc]{\footnotesize $0$}}
\put(-18,110){\makebox(0,0)[cc]{\footnotesize $0$}}
\put(15,97){\makebox(0,0)[cc]{\footnotesize $0$}}
\put(35,80){\makebox(0,0)[cc]{\footnotesize $0$}}
\put(-15,80){\makebox(0,0)[cc]{\footnotesize $0$}}
\put(18,80){\makebox(0,0)[cc]{\footnotesize $0$}}
\put(15,63){\makebox(0,0)[cc]{\footnotesize $-1$}}
\put(-20,40){\makebox(0,0)[cc]{\footnotesize $-1$}}
\put(5,40){\makebox(0,0)[cc]{\footnotesize $0$}}
\put(40,40){\makebox(0,0)[cc]{\footnotesize $-1$}}
\put(0,-10){\makebox(0,0)[cc]{$D_{\delta}$}}
\end{picture}} 
\end{picture}
\caption{Oriented virtual diagrams $D_{\alpha}$, $D_{\beta}$, $D_{\gamma}$, and $D_{\delta}$.} \label{fig7}
\end{figure}
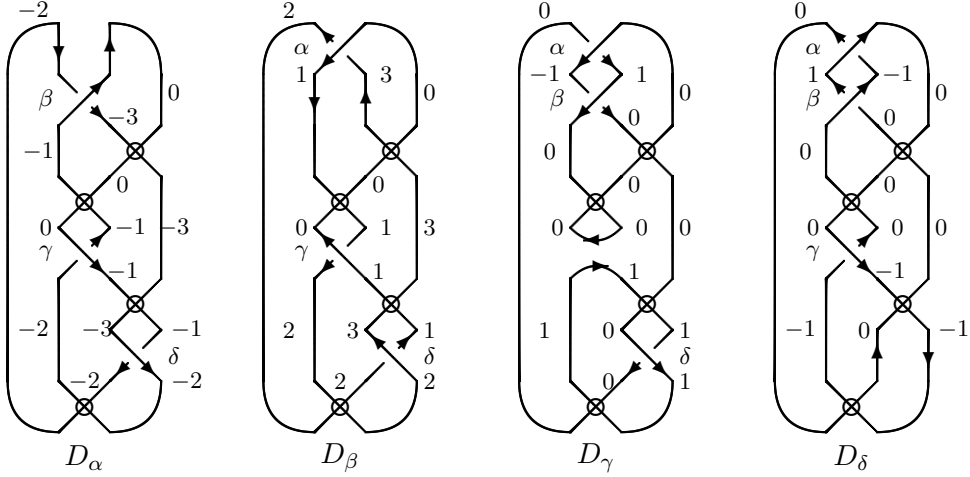

The result of calculations for oriented virtual knot diagrams $D_{\alpha}$, $D_{\beta}$, $D_{\gamma}$, and $D_{\delta}$, presented in Fig.~\ref{fig7}, is given in Table~\ref{t1}. 
\begin{table}[ht]
\caption{Values of $\operatorname{sgn}$, $\operatorname{Ind}$, and dwrithe for diagrams presented in Fig.~\ref{fig7}.}
\begin{tabular}{l | l | l|  l } \hline
 & Sign & index value & dwrihe \\ \hline 
$D_{\alpha}$ & $\operatorname{sgn}(\beta) = -1$ & $\operatorname{Ind}(\beta)=2$ & $\nabla J_{1}(D_{\alpha})= \operatorname{sgn}(\gamma) - \operatorname{sgn}(\delta) =  2$ \\ 
& $\operatorname{sgn}(\gamma)=1$ &  $\operatorname{Ind}(\gamma)=1$& $\nabla J_{2}(D_{\alpha}) = \operatorname{sgn} (\beta)=-1$ \\ 
& $\operatorname{sgn}(\delta) =-1$ & $\operatorname{Ind}(\delta)=-1$ &  \\ \hline  
$D_{\beta}$ & $\operatorname{sgn}(\alpha)=-1$ & $\operatorname{Ind}(\alpha)=-2$ & $\nabla J_{1}(D_{\beta})= \operatorname{sgn} (\delta) - \operatorname{sgn} (\gamma)=-2$ \\ 
& $\operatorname{sgn}(\gamma)=1$ & $\operatorname{Ind}(\gamma)=-1$ & $\nabla J_{2}(D_{\beta})= - \operatorname{sgn}(\alpha) = 1$ \\ 
& $\operatorname{sgn}(\delta) =-1$ &  $\operatorname{Ind}(\delta)=1$ &  \\ \hline  
$D_{\gamma}$ & $\operatorname{sgn}(\alpha)=1$ & $\operatorname{Ind}(\alpha)=-1$ &  $\nabla J_{1}(D_{\gamma})= \operatorname{sgn} (\beta) - \operatorname{sgn} (\alpha) = 0$ \\ 
& $\operatorname{sgn}(\beta)=1$ & $\operatorname{Ind}(\beta)=1$ & $\nabla J_{2}(D_{\gamma})=0$ \\ 
& $\operatorname{sgn}(\delta)=-1$ & $\operatorname{Ind}(\delta)=0$ &  \\ \hline  
$D_{\delta}$ & $\operatorname{sgn}(\alpha) =1$ & $\operatorname{Ind}(\alpha)=1$ & $\nabla J_{1}(D_{\delta})= \operatorname{sgn} (\alpha) - \operatorname{sgn} (\beta) = 0$ \\ 
& $ \operatorname{sgn}(\beta)=1$ & $\operatorname{Ind}(\beta)=-1$ & $\nabla J_{2}(D_{\delta}) =  0$ \\ 
& $\operatorname{sgn}(\gamma)=1$ & $\operatorname{Ind}(\gamma)=0$ &  \\ \hline  
\end{tabular} \label{t1}
\end{table} 
Basing on these calculations we obtain $L$-polynomials for diagram $D$. For $n = 1$ and $n=2$ we get 
$$
L^{1}_{D}(t,\ell)=2-t^{2}\ell^{2}-t^{-2}\ell^{2}, \qquad L^{2}_{D}(t,\ell)=2-t^{2}\ell-t^{-2}\ell.  
$$
For all $n \geq 3$ $L$-polynomials coincide with the affine polynomial:  
$$
L^{n}_{D}(t,\ell) = P_D(t) = 2-t^{2}-t^{-2}.
$$
This completes the Example~\ref{ex3.2}. 
}
\end{example}

The following result shows that $L$-polynomials are invariants of virtual knots. 

\begin{theorem} \label{thm-h}
Let $D$ be a diagram of an oriented virtual knot $K$. Then for any $n\in \mathbb{N}$, the polynomial $L^{n}_{D}(t,\ell)$ is an invariant for $K$. 
\end{theorem}

\begin{proof} 
Let $n$ be any positive integer. From the definition of $L$-polynomial it is clear that $L^{n}_{D}(t,\ell)$ is an invariant under virtual Reidemeister moves. Thus, it is enough to observe the behavior of $L^{n}_{D}(t,\ell)$ under classical Reidemeister moves RI, RII, RIII, and semi virtual move SV, see Fig.~\ref{fig-rei}. Let $D'$ be a diagram obtained from $D$ by applying one of these moves. We assume that in case of RI and RII, the number of classical crossings in $D'$ are more than the number of classical crossings in $D$. We have the following cases. 

\smallskip 

\underline{RI--move:} 
Let $D'$ be a diagram obtained from $D$ by move RI and $c'$ be the new crossing in $D'$. Fig.~\ref{figR1} presents all possible cases (i), (ii), (iii) and (iv), depending of orientation and crossing at $c'$.  Recall that $\operatorname{Ind} (c') = \operatorname{sgn} (c') (a-b-1)$, where labels $a$ and $b$ are presented in Fig.~\ref{fig3}.  In the case (i) it is shown in Fig.~\ref{figR1} that $b = a-1$, hence $\operatorname{Ind}(c')=0$. Similar considerations give that $\operatorname{Ind} (c') =0$ for all cases (i), (ii), (iii) and (iv).  
Also, it is clear from Fig.~\ref{figR1} that in cases (ii) and (iii) $D'_{c'}$ is equivalent to $D$, and in cases (i) and (iv) $D'_{c'}$ is equivalent to $D$ with the inverse orientation.  Hence, by  Lemma~\ref{lemma1} we get $\abs{\nabla J_{n}(D'_{c'})} = \abs{\nabla J_{n}(D)}$. By Remark~\ref{rem-inv}, $\nabla J_n (D') = \nabla J_n (D)$.  
\begin{figure}[!ht]
\centering 
\unitlength=0.44mm
\begin{picture}(0,65)(-10,5)
\thicklines
\put(0,40){\begin{picture}(0,00)
\put(-140,20){\makebox(0,0)[cc]{(i)}}
\put(-100,5){\makebox(0,0)[rc]{\footnotesize $a$}}
\put(-80,5){\makebox(0,0)[lc]{\footnotesize $a-1$}}
\qbezier(-130,10)(-130,10)(-130,30)
\put(-130,18){\vector(0,1){5}}
\put(-115,25){\makebox(0,0)[cc]{\footnotesize RI}}
\put(-115,20){\makebox(0,0)[cc]{$\longrightarrow$}}
\qbezier(-100,10)(-100,10)(-80,30)
\qbezier(-100,30)(-100,30)(-92,22) 
\qbezier(-80,10)(-80,10)(-88,18)
\qbezier(-80,30)(-70,20)(-80,10)
\put(-90,5){\makebox(0,0)[cc]{$c'$}}
\put(-95,25){\vector(-1,1){5}}
\put(-85,25){\vector(1,1){5}}
\put(-55,20){\makebox(0,0)[cc]{$\longrightarrow$}}
\qbezier(-40,30)(-30,20)(-20,30)
\qbezier(-40,10)(-30,20)(-20,10) 
\qbezier(-20,30)(-10,20)(-20,10)
\put(-25,25){\vector(1,1){5}}
\put(-35,15){\vector(-1,-1){5}}
\put(10,20){\makebox(0,0)[cc]{(ii)}}
\qbezier(20,10)(20,10)(20,30)
\put(20,22){\vector(0,-1){5}} 
\put(35,25){\makebox(0,0)[cc]{\footnotesize RI}}
\put(35,20){\makebox(0,0)[cc]{$\longrightarrow$}}
\qbezier(50,10)(50,10)(70,30)
\qbezier(50,30)(50,30)(58,22) 
\qbezier(70,10)(70,10)(62,18)
\qbezier(70,10)(80,20)(70,30)
\put(55,15){\vector(-1,-1){5}}
\put(65,15){\vector(1,-1){5}}
\put(60,5){\makebox(0,0)[cc]{$c'$}}
\put(90,20){\makebox(0,0)[cc]{$\longrightarrow$}}
\qbezier(100,30)(110,20)(120,30)
\qbezier(100,10)(110,20)(120,10) 
\qbezier(120,10)(130,20)(120,30)
\put(105,15){\vector(-1,-1){5}}
\put(115,25){\vector(1,1){5}}
\end{picture}}
%%%
\put(0,0){\begin{picture}(0,0) 
\put(-140,20){\makebox(0,0)[cc]{(iii)}}
\qbezier(-130,10)(-130,10)(-130,30)
\put(-130,18){\vector(0,1){5}}
\put(-115,25){\makebox(0,0)[cc]{\footnotesize RI}}
\put(-115,20){\makebox(0,0)[cc]{$\longrightarrow$}}
\qbezier(-100,10)(-100,10)(-92,18)
\qbezier(-80,30)(-80,30)(-88,22) 
\qbezier(-100,30)(-100,30)(-80,10)
\qbezier(-80,30)(-70,20)(-80,10)
\put(-95,25){\vector(-1,1){5}}
\put(-85,25){\vector(1,1){5}}
\put(-90,5){\makebox(0,0)[cc]{$c'$}}
\put(-55,20){\makebox(0,0)[cc]{$\longrightarrow$}}
\qbezier(-40,30)(-30,20)(-20,30)
\qbezier(-40,10)(-30,20)(-20,10) 
\qbezier(-20,30)(-10,20)(-20,10)
\put(-35,25){\vector(-1,1){5}}
\put(-25,15){\vector(1,-1){5}}
\put(10,20){\makebox(0,0)[cc]{(iv)}}
\qbezier(20,10)(20,10)(20,30)
\put(20,22){\vector(0,-1){5}} 
\put(35,25){\makebox(0,0)[cc]{\footnotesize RI}}
\put(35,20){\makebox(0,0)[cc]{$\longrightarrow$}}
\qbezier(50,30)(50,30)(70,10)
\qbezier(50,10)(50,10)(58,18) 
\qbezier(70,30)(70,30)(62,22)
\qbezier(70,10)(80,20)(70,30)
\put(55,15){\vector(-1,-1){5}}
\put(65,15){\vector(1,-1){5}}
\put(60,5){\makebox(0,0)[cc]{$c'$}}
\put(90,20){\makebox(0,0)[cc]{$\longrightarrow$}}
\qbezier(100,30)(110,20)(120,30)
\qbezier(100,10)(110,20)(120,10) 
\qbezier(120,10)(130,20)(120,30)
\put(105,25){\vector(-1,1){5}}
\put(115,15){\vector(1,-1){5}}
\end{picture}}
\end{picture}
\caption{RI-move and smoothing against orientation.} \label{figR1}
\end{figure}
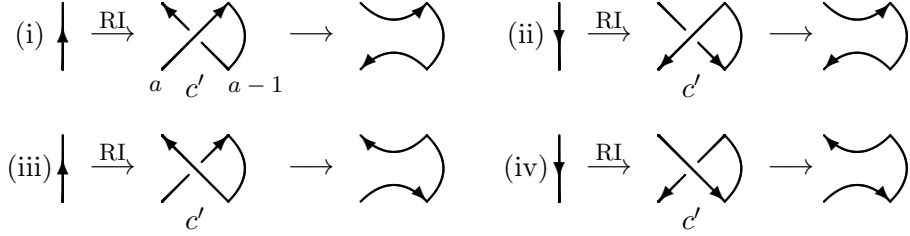

Therefore, 
\begin{eqnarray*}
L_{D'}^{n}(t,\ell) & = & \sum_{c \in C(D')} \operatorname{sgn}(c) \left(t^{\operatorname{Ind}(c)} \ell^{\abs{\nabla J_{n}(D'_{c})}} -\ell^{\abs{\nabla J_{n}(D')}}\right) \cr 
& = & \sum_{c \in C(D)} \operatorname{sgn}(c) \left( t^{\operatorname{Ind}(c)} \ell^{\abs{\nabla J_{n}(D_{c})}} -\ell^{\abs{\nabla J_{n}(D)}} \right) \cr 
& & +  \operatorname{sgn}(c') \left( t^{\operatorname{Ind}(c')}  \ell^{\abs{\nabla J_{n}(D'_{c'})}} - \ell^{\abs{\nabla J_{n}(D')}} \right)  \cr 
& = & L_{D}^{n}(t,\ell) +  \operatorname{sgn}(c') \left( t^{0}  \ell^{\abs{\nabla J_{n}(D)}} - \ell^{\abs{\nabla J_{n}(D)}} \right)= L_{D}^{n}(t,\ell).  
\end{eqnarray*}

\smallskip 

\underline{RII--move:}
Let $D'$ be a diagram obtained from $D$ by move RII and let $a$ and $b$ be new crossings in $D'$. Then $\operatorname{sgn}(a)=-\operatorname{sgn}(b)$ and $\operatorname{Ind}(a) = \operatorname{Ind}(b)$. Let $D'_a$ and $D'_b$ be diagrams, obtained from $D'$ by orientation reversing  smoothings at $a$ and $b$, respectively. We will consider two cases of moves RII. In Case (1), presented in Fig.~\ref{figR2a}, two arcs of $D$ have the same orientation; and in Case (2), presented in Fig.~\ref{figR2b}, two arcs of $D$ have different orientations.  
\begin{figure}[!ht]
\centering 
\unitlength=0.44mm
\begin{picture}(0,110)(-10,0)
\thicklines
\put(0,60){\begin{picture}(0,0)
\qbezier(-100,-20)(-100,-20)(-100,20)
\qbezier(-80,-20)(-80,-20)(-80,20)
\put(-100,-2){\vector(0,1){5}}
\put(-80,-2){\vector(0,1){5}}
\put(-90,-30){\makebox(0,0)[cc]{$D$}}
\put(-60,5){\makebox(0,0)[cc]{\footnotesize RII}}
\put(-60,0){\makebox(0,0)[cc]{$\longrightarrow$}}
\qbezier(-40,20)(-40,20)(-20,0)
\qbezier(-40,0)(-40,0)(-32,8) 
\qbezier(-20,20)(-20,20)(-28,12)
\qbezier(-40,-20)(-40,-20)(-20,0)
\qbezier(-40,0)(-40,0)(-32,-8) 
\qbezier(-20,-20)(-20,-20)(-28,-12)
\put(-35,15){\vector(-1,1){5}}
\put(-25,15){\vector(1,1){5}}
\put(-35,-5){\vector(-1,1){5}}
\put(-25,-5){\vector(1,1){5}}
\put(-40,10){\makebox(0,0)[cc]{$a$}}
\put(-40,-10){\makebox(0,0)[cc]{$b$}}
\put(-30,-30){\makebox(0,0)[cc]{$D'$}}
\put(0,30){\makebox(0,0)[cc]{$\longrightarrow$}}
\put(0,-30){\makebox(0,0)[cc]{$\longrightarrow$}}
\qbezier(20,50)(30,40)(40,50)
\qbezier(20,30)(30,40)(40,30) 
\qbezier(20,10)(20,10)(40,30) 
\qbezier(20,30)(20,30)(28,22) 
\qbezier(40,10)(40,10)(32,18)
\put(32,45){\vector(-1,0){5}}
\put(28,35){\vector(1,0){5}}
\put(30,0){\makebox(0,0)[cc]{$D'_a$}}
\put(60,35){\makebox(0,0)[cc]{RI}}
\put(60,30){\makebox(0,0)[cc]{$\longrightarrow$}}
\qbezier(80,50)(80,50)(80,40)
\qbezier(80,40)(90,30)(100,40)
\qbezier(100,40)(100,40)(100,50)
\qbezier(80,10)(80,10)(80,20)
\qbezier(80,20)(90,30)(100,20)
\qbezier(100,20)(100,20)(100,10)
\put(92,35){\vector(-1,0){5}}
\put(92,25){\vector(-1,0){5}} 
\put(90,0){\makebox(0,0)[cc]{$D''_a$}}
\end{picture}}
%%%
\put(0,0){\begin{picture}(0,0) 
\qbezier(20,50)(20,50)(40,30)
\qbezier(20,30)(20,30)(28,38) 
\qbezier(40,50)(40,50)(32,42) 
\qbezier(20,30)(30,20)(40,30) 
\qbezier(20,10)(30,20)(40,10)
\put(28,25){\vector(1,0){5}}
\put(32,15){\vector(-1,0){5}}
\put(30,0){\makebox(0,0)[cc]{$D'_b$}}
\put(60,35){\makebox(0,0)[cc]{RI}}
\put(60,30){\makebox(0,0)[cc]{$\longrightarrow$}}
\qbezier(80,50)(80,50)(80,40)
\qbezier(80,40)(90,30)(100,40)
\qbezier(100,40)(100,40)(100,50)
\qbezier(80,10)(80,10)(80,20)
\qbezier(80,20)(90,30)(100,20)
\qbezier(100,20)(100,20)(100,10)
\put(92,35){\vector(-1,0){5}}
\put(92,25){\vector(-1,0){5}} 
\put(90,0){\makebox(0,0)[cc]{$D''_b$}}
\end{picture}}
\end{picture}
\caption{RII-move and smoothing against orientation. Case (1).} \label{figR2a}
\end{figure}
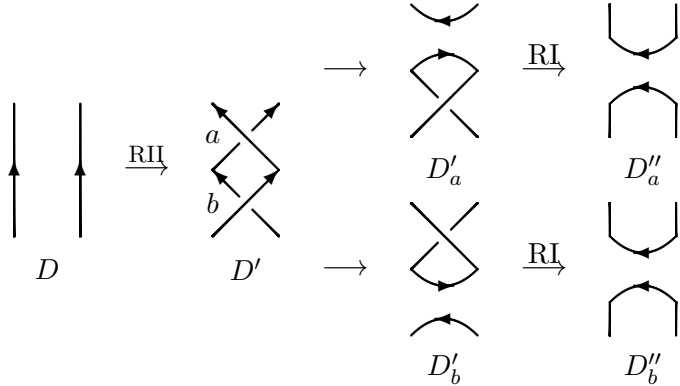

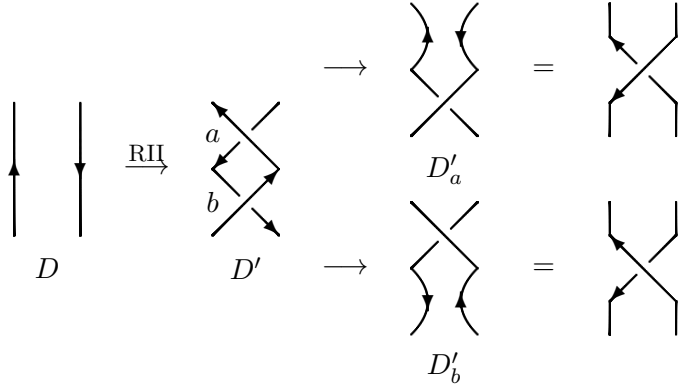
\begin{figure}[!ht]
\centering 
\unitlength=0.44mm
\begin{picture}(0,110)(-10,0)
\thicklines
\put(0,60){\begin{picture}(0,0)
\qbezier(-100,-20)(-100,-20)(-100,20)
\qbezier(-80,-20)(-80,-20)(-80,20)
\put(-100,-2){\vector(0,1){5}}
\put(-80,2){\vector(0,-1){5}}
\put(-90,-30){\makebox(0,0)[cc]{$D$}}
\put(-60,5){\makebox(0,0)[cc]{\footnotesize RII}}
\put(-60,0){\makebox(0,0)[cc]{$\longrightarrow$}}
\qbezier(-40,20)(-40,20)(-20,0)
\qbezier(-40,0)(-40,0)(-32,8) 
\qbezier(-20,20)(-20,20)(-28,12)
\qbezier(-40,-20)(-40,-20)(-20,0)
\qbezier(-40,0)(-40,0)(-32,-8) 
\qbezier(-20,-20)(-20,-20)(-28,-12)
\put(-35,15){\vector(-1,1){5}}
\put(-35,5){\vector(-1,-1){5}}
\put(-25,-5){\vector(1,1){5}}
\put(-25,-15){\vector(1,-1){5}}
\put(-40,10){\makebox(0,0)[cc]{$a$}}
\put(-40,-10){\makebox(0,0)[cc]{$b$}}
\put(-30,-30){\makebox(0,0)[cc]{$D'$}}
\put(0,30){\makebox(0,0)[cc]{$\longrightarrow$}}
\put(0,-30){\makebox(0,0)[cc]{$\longrightarrow$}}
\qbezier(20,50)(30,40)(20,30)
\qbezier(40,50)(30,40)(40,30) 
\qbezier(20,10)(20,10)(40,30) 
\qbezier(20,30)(20,30)(28,22) 
\qbezier(40,10)(40,10)(32,18)
\put(25,38){\vector(0,1){5}}
\put(35,42){\vector(0,-1){5}}
\put(30,0){\makebox(0,0)[cc]{$D'_a$}}
\put(60,30){\makebox(0,0)[cc]{$=$}}
\qbezier(80,50)(80,50)(80,40)
\qbezier(80,40)(80,40)(88,32)
\qbezier(100,20)(100,20)(92,28)
\qbezier(100,40)(100,40)(100,50)
\qbezier(80,10)(80,10)(80,20)
\qbezier(80,20)(80,20)(100,40)
\qbezier(100,20)(100,20)(100,10)
\put(85,25){\vector(-1,-1){5}}
\put(85,35){\vector(-1,1){5}} 
\end{picture}}
%%%
\put(0,0){\begin{picture}(0,0) 
\qbezier(20,50)(20,50)(40,30)
\qbezier(20,30)(20,30)(28,38) 
\qbezier(40,50)(40,50)(32,42) 
\qbezier(20,30)(30,20)(20,10) 
\qbezier(40,30)(30,20)(40,10)
\put(25,22){\vector(0,-1){5}}
\put(35,18){\vector(0,1){5}}
\put(30,0){\makebox(0,0)[cc]{$D'_b$}}
\put(60,30){\makebox(0,0)[cc]{$=$}}
\qbezier(80,50)(80,50)(80,40)
\qbezier(80,40)(80,40)(100,20)
\qbezier(100,40)(100,40)(100,50)
\qbezier(80,10)(80,10)(80,20)
\qbezier(80,20)(80,20)(88,28)
\qbezier(100,40)(100,40)(92,32)
\qbezier(100,20)(100,20)(100,10)
\put(85,35){\vector(-1,1){5}}
\put(85,25){\vector(-1,-1){5}} 
\end{picture}}
\end{picture}
\caption{RII-move and smoothing against orientation. Case (2).} \label{figR2b}
\end{figure}

In Case (1) diagram $D'_a$ is RI--equivalent to diagram $D''_a$, and diagram $D'_b$ is RI-equivalent to diagram $D''_b$. Since $n$-dwrithe is RI-invariant and diagrams $D''_a$ and $D''_b$ coincide, we get $\nabla J_n (D'_a) = \nabla J_n (D'_b)$.  

In Case (2) we get oriented diagrams $D'_a$ and $D'_b$ that have difference only in one crossing point. Since by Lemma~\ref{lemma1}, $n$-th writhe is a flat virtual knot invariant, we also get $\nabla J_n(D'_a) = \nabla J_n (D'_b)$. 

Therefore, in both cases we have  
\begin{eqnarray*}
L_{D'}^{n}(t,\ell) & = & L_{D}^{n}(t,\ell) + \operatorname{sgn}(a) \left(t^{\operatorname{Ind}(a)}\ell^{\abs{\nabla J_{n}(D'_{a})}} -\ell^{\abs{\nabla J_{n}(D')}} \right) \cr  
& & + \operatorname{sgn}(b) \left(t^{\operatorname{Ind}(b)} \ell^{\abs{\nabla J_{n}(D'_{b})}} - \ell^{\abs{\nabla J_{n}(D')}} \right) \cr 
& = & L_{D}^{n}(t,\ell) + \operatorname{sgn}(a)  \left(t^{\operatorname{Ind}(a)}\ell^{\abs{\nabla J_{n}(D'_{a})}}  - \ell^{\abs{\nabla J_{n}(D')}} \right) \cr 
& & - \operatorname{sgn} (a) \left(t^{\operatorname{Ind}(a)}\ell^{\abs{\nabla J_{n}(D'_{a})}}  - \ell^{\abs{\nabla J_{n}(D')}} \right) = L_D^n(t, \ell). 
\end{eqnarray*} 

For remaining cases of RII-moves with another over/under crossings the same arguments give the invariance of $L^n_D (t, \ell)$.   

\smallskip 

\underline{RIII--move:}
Let $D'$ be a diagram obtained form $D$ by RIII-move. Let $a$, $b$ and $c$ be crossings in $D$, and  $a'$, $b'$ and $c'$ be corresponding crossings of $D'$, as shown in~Fig.~\ref{figR3}.   
\begin{figure}[!ht]
\centering 
\unitlength=0.5mm
\begin{picture}(0,35)(0,5)
\thicklines
\qbezier(-50,10)(-50,10)(-30,30)
\qbezier(-50,30)(-50,30)(-42,22) 
\qbezier(-30,10)(-30,10)(-38,18)
\qbezier(-60,32)(-60,32)(-20,32)
\qbezier(-54,34)(-54,34)(-60,40)
\qbezier(-26,34)(-26,34)(-20,40)
\put(-54,27){\makebox(0,0)[cc]{\footnotesize $a$}}
\put(-26,27){\makebox(0,0)[cc]{\footnotesize $b$}}
\put(-40,12){\makebox(0,0)[cc]{\footnotesize $c$}}
\put(0,25){\makebox(0,0)[cc]{$\longleftrightarrow$}}
\qbezier(50,40)(50,40)(30,20)
\qbezier(50,20)(50,20)(42,28) 
\qbezier(30,40)(30,40)(38,32)
\qbezier(20,18)(20,18)(60,18)
\qbezier(20,10)(20,10)(26,16)
\qbezier(60,10)(60,10)(54,16)
\put(54,23){\makebox(0,0)[cc]{\footnotesize $a'$}}
\put(26,23){\makebox(0,0)[cc]{\footnotesize $b'$}}
\put(40,38){\makebox(0,0)[cc]{\footnotesize $c'$}}
\end{picture}
\caption{RIII move.} \label{figR3}
\end{figure}
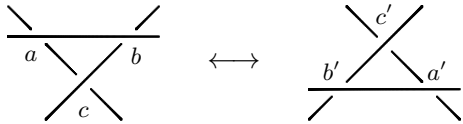

Because, of a freedom to choose orientation for each of three arcs, there are eight case for RIII-moves. We consider one of them in details. Assume that orientations of arcs of $D$ are as in Fig.~\ref{figR3case1} and of arcs of $D'$ are as in Fig.~\ref{figR3case2}.  
\begin{figure}[!ht]
\centering 
\unitlength=0.5mm
\begin{picture}(0,35)(-10,5)
\thicklines
\put(-60,0){\begin{picture}(0,30) 
\qbezier(-50,10)(-50,10)(-30,30)
\qbezier(-50,30)(-50,30)(-42,22) 
\qbezier(-30,10)(-30,10)(-38,18)
\qbezier(-60,32)(-60,32)(-20,32)
\qbezier(-54,34)(-54,34)(-60,40)
\qbezier(-26,34)(-26,34)(-20,40)
\put(-54,27){\makebox(0,0)[cc]{\footnotesize $a$}}
\put(-26,27){\makebox(0,0)[cc]{\footnotesize $b$}}
\put(-40,12){\makebox(0,0)[cc]{\footnotesize $c$}}
\put(-25,32){\vector(1,0){5}}
\put(-55,35){\vector(-1,1){5}}
\put(-25,35){\vector(1,1){5}}
\put(-40,5){\makebox(0,0)[cc]{$D$}}
\put(-10,25){\makebox(0,0)[cc]{$\longrightarrow$}}
\end{picture}}
\put(0,0){\begin{picture}(0,30) 
\qbezier(-50,10)(-50,10)(-30,30)
\qbezier(-30,10)(-30,10)(-38,18)
\qbezier(-60,32)(-50,32)(-42,22)
\qbezier(-40,32)(-40,32)(-20,32)
\qbezier(-40,32)(-50,32)(-60,40)
\qbezier(-26,34)(-26,34)(-20,40)
\put(-25,32){\vector(1,0){5}}
\put(-55,32){\vector(-1,0){5}}
\put(-40,5){\makebox(0,0)[cc]{$D_a$}}
\end{picture}}
\put(60,0){\begin{picture}(0,30) 
\qbezier(-50,10)(-50,10)(-35,25)
\qbezier(-50,30)(-50,30)(-42,22) 
\qbezier(-20,40)(-35,33)(-20, 32) 
\qbezier(-30,10)(-30,10)(-38,18)
\qbezier(-60,32)(-60,32)(-35,32)
\qbezier(-54,34)(-54,34)(-60,40)
\qbezier(-35,25)(-30,30)(-35,32)
\put(-20,32){\vector(1,0){5}}
\put(-55,32){\vector(-1,0){5}}
\put(-40,5){\makebox(0,0)[cc]{$D_b$}}
\end{picture}}
\put(120,0){\begin{picture}(0,30) 
\qbezier(-50,10)(-40,20)(-30,10)
\qbezier(-50,30)(-40,20)(-30,30) 
\qbezier(-60,32)(-60,32)(-20,32)
\qbezier(-54,34)(-54,34)(-60,40)
\qbezier(-26,34)(-26,34)(-20,40)
\put(-42,25){\vector(1,0){5}}
\put(-42,15){\vector(1,0){5}}
\put(-40,5){\makebox(0,0)[cc]{$D_c$}}
\end{picture}}
\end{picture}
\caption{RIII move and orientation reversing smoothing. Case (1).} \label{figR3case1}
\end{figure}
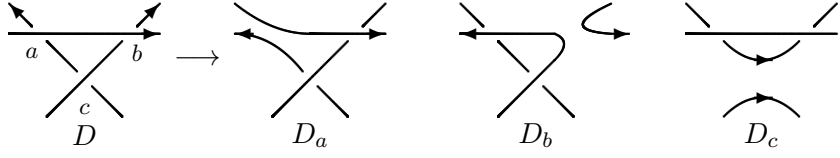

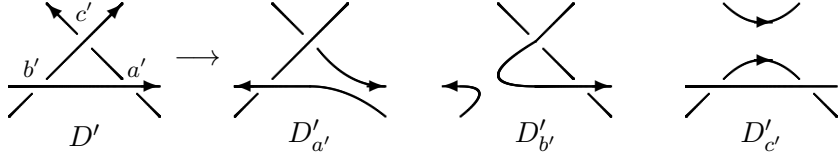
\begin{figure}[!ht]
\centering 
\unitlength=0.5mm
\begin{picture}(0,35)(10,5)
\thicklines
\put(-120,0){\begin{picture}(0,30) 
\qbezier(50,40)(50,40)(30,20)
\qbezier(50,20)(50,20)(42,28) 
\qbezier(30,40)(30,40)(38,32)
\qbezier(20,18)(20,18)(60,18)
\qbezier(20,10)(20,10)(26,16)
\qbezier(60,10)(60,10)(54,16)
\put(54,23){\makebox(0,0)[cc]{\footnotesize $a'$}}
\put(26,23){\makebox(0,0)[cc]{\footnotesize $b'$}}
\put(40,38){\makebox(0,0)[cc]{\footnotesize $c'$}}
\put(55,18){\vector(1,0){5}}
\put(45,35){\vector(1,1){5}}
\put(35,35){\vector(-1,1){5}}
\put(40,5){\makebox(0,0)[cc]{$D'$}}
\put(70,25){\makebox(0,0)[cc]{$\longrightarrow$}}
\end{picture}}
\put(-60,0){\begin{picture}(0,30) 
\qbezier(50,40)(50,40)(30,20)
\qbezier(60,18)(50,18)(42,28) 
\qbezier(30,40)(30,40)(38,32)
\qbezier(20,18)(20,18)(40,18)
\qbezier(40,18)(50,18)(60,10) 
\qbezier(20,10)(20,10)(26,16)
\put(55,18){\vector(1,0){5}}
\put(25,18){\vector(-1,0){5}}
\put(40,5){\makebox(0,0)[cc]{$D'_{a'}$}}
\end{picture}}
\put(0,0){\begin{picture}(0,30) 
\qbezier(50,40)(50,40)(40,30)
\qbezier(50,20)(50,20)(42,28) 
\qbezier(30,40)(30,40)(38,32)
\qbezier(40,18)(40,18)(60,18)
\qbezier(20,10)(30,18)(20,18)
\qbezier(40,30)(20,18)(40,18)
\qbezier(60,10)(60,10)(54,16)
\put(55,18){\vector(1,0){5}}
\put(20,18){\vector(-1,0){5}}
\put(40,5){\makebox(0,0)[cc]{$D'_{b'}$}}
\end{picture}}
\put(60,0){\begin{picture}(0,30) 
\qbezier(50,40)(40,30)(30,40)
\qbezier(50,20)(40,30)(30,20) 
\qbezier(20,18)(20,18)(60,18)
\qbezier(20,10)(20,10)(26,16)
\qbezier(60,10)(60,10)(54,16)
\put(38,35){\vector(1,0){5}}
\put(38,25){\vector(1,0){5}}
\put(40,5){\makebox(0,0)[cc]{$D'_{c'}$}}
\end{picture}}
\end{picture}
\caption{RIII move and orientation revising smoothing. Case (2).} \label{figR3case2}
\end{figure}

It is easy to see that $\operatorname{sgn} (a') = \operatorname{sgn}(a)$, $\operatorname{sgn}(b') = \operatorname{sgn}(b)$, and $\operatorname{sgn}(c') = \operatorname{sgn}(c)$. Moreover, direct calculations of Cheng coloring labels, by the rule presented in Fig.~\ref{fig3}, imply that $\operatorname{Ind} (a') = \operatorname{Ind} (a)$, $\operatorname{Ind} (b') = \operatorname{Ind} (b)$, and $\operatorname{Ind}(c') = \operatorname{Ind} (c)$. Therefore, $\nabla J_n (D') = \nabla J_n (D)$. 

Now we consider orientations reversing smoothings of $D$ at crossings $a$, $b$, $c$ and of $D'$ at crossings $a'$, $b'$, $c'$. By comparing Fig.~\ref{figR3case1} and Fig.~\ref{figR3case2} we see that $D'_{a'}$ and $D_a$ are equivalent under two crossing change operations, hence by Lemma~\ref{lemma1} we have $\nabla J_n (D'_{a'}) = \nabla J_n (D_a)$. Diagram $D'_{b'}$ is equivalent to $D_b$ under two RII-moves, and analogously, diagram $D'_{c'}$ is equivalent to $D_c$ under two RII moves. Therefore, by Remark~\ref{rem-inv} we have  $\nabla J_n (D'_{b'}) = \nabla J_n (D_b)$ and $\nabla J_n (D'_{c'}) = \nabla J_n (D_c)$. 

By the above discussion, 
$$
\begin{gathered}
\operatorname{sgn}(a') \left(t^{\operatorname{Ind}(a')}\ell^{\abs{\nabla J_{n}(D'_{a'})}} -\ell^{\abs{\nabla J_{n}(D')}} \right) + \operatorname{sgn}(b') \left(t^{\operatorname{Ind}(b')}\ell^{\abs{\nabla J_{n}(D'_{b'})}} -\ell^{\abs{\nabla J_{n}(D')}} \right) \\ 
+ \operatorname{sgn}(c') \left(t^{\operatorname{Ind}(c')}\ell^{\abs{\nabla J_{n}(D'_{c'})}} -\ell^{\abs{\nabla J_{n}(D')}} \right) \\ 
=  \operatorname{sgn}(a) \left(t^{\operatorname{Ind}(a)}\ell^{\abs{\nabla J_{n}(D_{a})}} -\ell^{\abs{\nabla J_{n}(D)}} \right) 
+ \operatorname{sgn}(b) \left(t^{\operatorname{Ind}(b)}\ell^{\abs{\nabla J_{n}(D_{b})}} -\ell^{\abs{\nabla J_{n}(D)}} \right) \\ + \operatorname{sgn}(c) \left(t^{\operatorname{Ind}(c)}\ell^{\abs{\nabla J_{n}(D_{c})}} -\ell^{\abs{\nabla J_{n}(D)}} \right). 
\end{gathered} 
$$ 
Therefore, $L_{D'}^{n}(t,\ell) = L_{D}^{n}(t,\ell)$. 
For other types of RIII-moves the result follows by analogous considerations. 

\smallskip 

\underline{SV--move:} 
Let $D'$ be a diagram, obtained from $D$ by SV-move applied at classical crossing $c$, and $c'$ be the correspond crossing of $D'$. We will discuss two cases depending of arc orientations. 
\begin{figure}[!ht]
\centering 
\unitlength=0.5mm
\begin{picture}(0,45)(0,0)
\thicklines
\qbezier(-120,30)(-120,30)(-80,30)
\qbezier(-110,10)(-100,20)(-90,10)
\qbezier(-110,30)(-100,20)(-90,30)
\qbezier(-120,40)(-120,40)(-110,30)
\qbezier(-80,40)(-80,40)(-90,30)
\put(-110,30){\circle{4}}
\put(-90,30){\circle{4}}
\put(-98,25){\vector(-1,0){5}}
\put(-102,15){\vector(1,0){5}}
\put(-100,0){\makebox(0,0)[cc]{$D_c$}}
\put(-70,25){\makebox(0,0)[cc]{$\longleftarrow$}}
\qbezier(-50,10)(-50,10)(-42,18)
\qbezier(-20,40)(-20,40)(-38,22)
\qbezier(-60,40)(-60,40)(-30,10) 
\qbezier(-60,30)(-60,30)(-20,30)
\put(-50,30){\circle{4}}
\put(-30,30){\circle{4}}
\put(-40,20){\vector(-1,1){6}}
\put(-37,23){\vector(1,1){3}}
\put(-40,12){\makebox(0,0)[cc]{\footnotesize $c$}}
\put(-40,0){\makebox(0,0)[cc]{$D$}}
\put(0,25){\makebox(0,0)[cc]{$\longleftrightarrow$}}
\qbezier(20,20)(20,20)(60,20)
\qbezier(30,40)(30,40)(60,10)
\qbezier(20,10)(20,10)(38,28) 
\qbezier(50,40)(50,40)(42,32)
\put(30,20){\circle{4}}
\put(50,20){\circle{4}}
\put(35,35){\vector(-1,1){5}}
\put(45,35){\vector(1,1){5}}
\put(40,38){\makebox(0,0)[cc]{\footnotesize $c'$}}
\put(40,0){\makebox(0,0)[cc]{$D'$}}
\put(70,25){\makebox(0,0)[cc]{$\longrightarrow$}}
\put(110,20){\circle{4}}
\put(90,20){\circle{4}}
\put(102,35){\vector(-1,0){5}}
\put(98,25){\vector(1,0){5}}
\qbezier(80,20)(80,20)(120,20)
\qbezier(90,20)(100,30)(110,20)
\qbezier(90,40)(100,30)(110,40)
\qbezier(80,10)(80,10)(90,20)
\qbezier(120,10)(120,10)(110,20)
\put(100,0){\makebox(0,0)[cc]{$D'_{c'}$}}
\end{picture}
\caption{SV-move. Case (1).} \label{figSVcase1}
\end{figure}
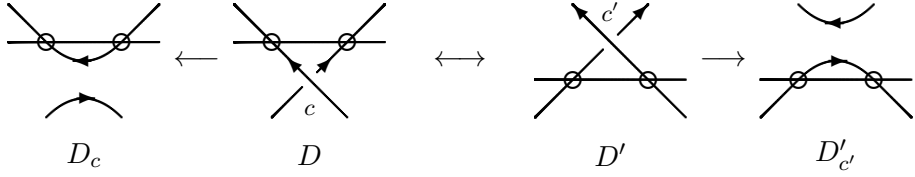

In Case (1), presented in Fig.~\ref{figSVcase1},  we see that $\operatorname{sgn} (c') = \operatorname{sgn} (c)$ and $\operatorname{Ind} (c') = \operatorname{Ind}(c)$. Hence $\nabla J_n (D') = \nabla J_n (D)$. As we see from Fig.~\ref{figSVcase1}, diagrams $D'_{c'}$ and $D_c$ are equivalent under two VRII-moves. Then, by Remark~\ref{rem-inv} we have $\nabla J_n (D'_{c'}) = \nabla J_n (D_c)$.  

\begin{figure}[!ht]
\centering 
\unitlength=0.5mm
\begin{picture}(0,45)(0,0)
\thicklines
\qbezier(-120,30)(-120,30)(-80,30)
\qbezier(-110,10)(-100,20)(-110,30)
\qbezier(-90,10)(-100,20)(-90,30)
\qbezier(-120,40)(-120,40)(-110,30)
\qbezier(-80,40)(-80,40)(-90,30)
\put(-110,30){\circle{4}}
\put(-90,30){\circle{4}}
\put(-105,18){\vector(0,1){5}}
\put(-95,18){\vector(0,1){5}}
\put(-100,0){\makebox(0,0)[cc]{$D_c$}}
\put(-70,25){\makebox(0,0)[cc]{$\longleftarrow$}}
\qbezier(-50,10)(-50,10)(-42,18)
\qbezier(-20,40)(-20,40)(-38,22)
\qbezier(-60,40)(-60,40)(-30,10) 
\qbezier(-60,30)(-60,30)(-20,30)
\put(-50,30){\circle{4}}
\put(-30,30){\circle{4}}
\put(-40,20){\vector(-1,1){6}}
\put(-43,17){\vector(-1,-1){3}}
\put(-40,12){\makebox(0,0)[cc]{\footnotesize $c$}}
\put(-40,0){\makebox(0,0)[cc]{$D$}}
\put(0,25){\makebox(0,0)[cc]{$\longleftrightarrow$}}
\qbezier(20,20)(20,20)(60,20)
\qbezier(30,40)(30,40)(60,10)
\qbezier(20,10)(20,10)(38,28) 
\qbezier(50,40)(50,40)(42,32)
\put(30,20){\circle{4}}
\put(50,20){\circle{4}}
\put(40,30){\vector(-1,1){6}}
\put(37,27){\vector(-1,-1){3}}
\put(40,38){\makebox(0,0)[cc]{\footnotesize $c'$}}
\put(40,0){\makebox(0,0)[cc]{$D'$}}
\put(70,25){\makebox(0,0)[cc]{$\longrightarrow$}}
\put(110,20){\circle{4}}
\put(90,20){\circle{4}}
\put(95,28){\vector(0,1){5}}
\put(105,32){\vector(0,-1){5}}
\qbezier(80,20)(80,20)(120,20)
\qbezier(90,20)(100,30)(90,40)
\qbezier(110,20)(100,30)(110,40)
\qbezier(80,10)(80,10)(90,20)
\qbezier(120,10)(120,10)(110,20)
\put(100,0){\makebox(0,0)[cc]{$D'_{c'}$}}
\end{picture}
\caption{SV-move. Case (2).} \label{figSVcase2}
\end{figure}
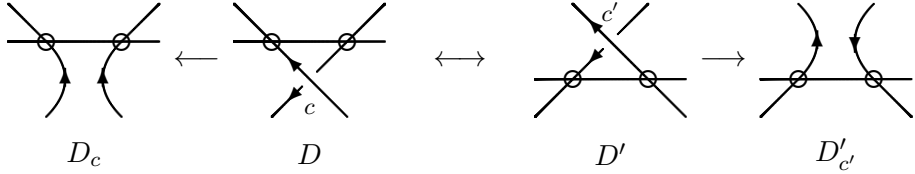

Analogously, in Case (2), presented in Fig.~\ref{figSVcase2}, we see that $\operatorname{sgn} (c') = \operatorname{sgn} (c)$ and $\operatorname{Ind} (c') = \operatorname{Ind}(c)$. Hence $\nabla J_n (D') = \nabla J_n (D)$. As we see from Fig.~\ref{figSVcase2}, diagrams $D'_{c'}$ and $D_c$ are equivalent, so $\nabla J_n (D'_{c'}) = \nabla J_n (D_c)$. 

Thus, in Case (1) as well as in Case (2) we get 
$$
 \operatorname{sgn}(c') \left(t^{\operatorname{Ind}(c')}\ell^{\abs{\nabla J_{n}(D'_{c'})}} -\ell^{\abs{\nabla J_{n}(D')}} \right)  =  \operatorname{sgn}(c) \left(t^{\operatorname{Ind}(c)}\ell^{\abs{\nabla J_{n}(D_{c})}} -\ell^{\abs{\nabla J_{n}(D)}} \right). 
$$ 
Therefore, $L^{n}_{D'}(t,\ell) = L^{n}_{D}(t,\ell)$. Applying analogous arguments for other types of SV-moves, we conclude that $L^{n}_{D}(t,\ell)$ is invariant under SV-moves. 

All considered cases of moves RI, RII, RIII, and SV give that $L^{n}_{D}(t,\ell)$ is a virtual knot invariant.
\end{proof} 

\begin{proposition} \label{prop3.4}
The $L$-polynomials and the affine index polynomial coincide on classical knots. 
\end{proposition}

\begin{proof} 
Let $D$ be a diagram of a classical knot $K$ and $n \in \mathbb{N}$. Since dwrithe is a flat virtual knot invariant, $\nabla J_{n}(D)$ and $\nabla J_{n}(D_{c})$ equal to zero for any $c \in C(D)$. Thus $L^{n}_{D}(t,\ell)=P_{D}(t)$. 
\end{proof} 

\begin{remark} {\rm 
For $\ell=1$ we get $L^{n}_{K}(t,1)=P_{K}(t)$. Thus the affine index polynomial becomes a special case of the $L$-polynomial. 
}  
 \end{remark}

\begin{example}  \label{example3.6} {\rm 
Consider oriented virtual knot $K$ and its mirror image $K^*$ presented by the diagrams in Fig.~\ref{fig13}. 
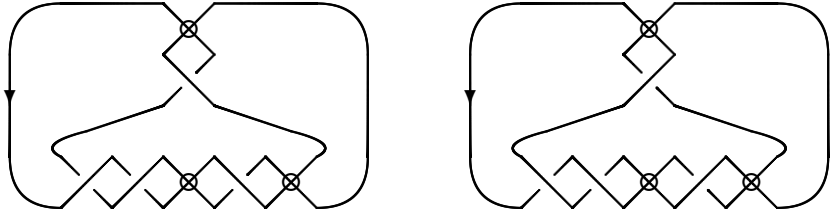
\begin{figure}[!ht]
\centering 
\unitlength=0.34mm
\begin{picture}(0,85)
\put(-90,0){\begin{picture}(0,80) 
\thicklines
\qbezier(-50,0)(-50,0)(-30,20)
\qbezier(-50,20)(-50,20)(-43,13) 
\qbezier(-30,0)(-30,0)(-37,7)
\qbezier(-30,0)(-30,0)(-10,20)
\qbezier(-30,20)(-30,20)(-23,13) 
\qbezier(-10,0)(-10,0)(-17,7)
\qbezier(-10,0)(-10,0)(10,20)
\qbezier(-10,20)(-10,20)(10,0) 
\put(0,10){\circle{6}}
\qbezier(10,20)(10,20)(30,0)
\qbezier(10,0)(10,0)(17,7) 
\qbezier(23,13)(23,13)(30,20)
\qbezier(30,0)(30,0)(50,20)
\qbezier(30,20)(30,20)(50,0) 
\put(40,10){\circle{6}}
\qbezier(-10,40)(-10,40)(-3,47)
\qbezier(3,53)(3,53)(10,60) 
\qbezier(-10,60)(-10,60)(10,40)
\qbezier(-10,60)(-10,60)(10,80)
\qbezier(-10,80)(-10,80)(10,60) 
\put(0,70){\circle{6}}
\qbezier(-50,20)(-60,25)(-40,30)
\qbezier(-40,30)(-10,40)(-10,40)
\qbezier(50,20)(60,25)(40,30)
\qbezier(40,30)(10,40)(10,40)
\qbezier(-50,0)(-70,0)(-70,20)
\qbezier(-70,20)(-70,20)(-70,60)
\qbezier(-70,60)(-70,80)(-50,80) 
\qbezier(-50,80)(-50,80)(-10,80) 
\qbezier(50,0)(70,0)(70,20)
\qbezier(70,20)(70,20)(70,60)
\qbezier(70,60)(70,80)(50,80) 
\qbezier(50,80)(50,80)(10,80) 
\put(-70,45){\vector(0,-1){5}}
\end{picture}}
%%%
\put(90,0){\begin{picture}(0,80) 
\thicklines
\qbezier(-50,20)(-50,20)(-30,0)
\qbezier(-50,0)(-50,0)(-43,7) 
\qbezier(-30,20)(-30,20)(-37,13)
\qbezier(-30,20)(-30,20)(-10,0)
\qbezier(-30,0)(-30,0)(-23,7) 
\qbezier(-10,20)(-10,20)(-17,13)
\qbezier(-10,0)(-10,0)(10,20)
\qbezier(-10,20)(-10,20)(10,0) 
\put(0,10){\circle{6}}
\qbezier(10,0)(10,0)(30,20)
\qbezier(10,20)(10,20)(17,13) 
\qbezier(23,7)(23,7)(30,0)
\qbezier(30,0)(30,0)(50,20)
\qbezier(30,20)(30,20)(50,0) 
\put(40,10){\circle{6}}
\qbezier(-10,60)(-10,60)(-3,53)
\qbezier(3,47)(3,47)(10,40) 
\qbezier(-10,40)(-10,40)(10,60)
\qbezier(-10,60)(-10,60)(10,80)
\qbezier(-10,80)(-10,80)(10,60) 
\put(0,70){\circle{6}}
\qbezier(-50,20)(-60,25)(-40,30)
\qbezier(-40,30)(-10,40)(-10,40)
\qbezier(50,20)(60,25)(40,30)
\qbezier(40,30)(10,40)(10,40)
\qbezier(-50,0)(-70,0)(-70,20)
\qbezier(-70,20)(-70,20)(-70,60)
\qbezier(-70,60)(-70,80)(-50,80) 
\qbezier(-50,80)(-50,80)(-10,80) 
\qbezier(50,0)(70,0)(70,20)
\qbezier(70,20)(70,20)(70,60)
\qbezier(70,60)(70,80)(50,80) 
\qbezier(50,80)(50,80)(10,80) 
\put(-70,45){\vector(0,-1){5}}
\end{picture}}
\end{picture}
\caption{Oriented virtual knots $K$ (on left) and $K^*$ (on right).} \label{fig13}
\end{figure}
The  affine index polynomial and the writhe polynomial are trivial for these knots, while $L$-polynomials are non trivial:  
$$
L^{1}_{K}(t,\ell) = t^{-1} \ell^{2}+t \ell^{2}-t^{-1}-t, \qquad   L^{2}_{K}(t,\ell)  =   t^{-1}\ell + t\ell - t^{-1} - t, 
$$
and 
$$
L^{1}_{K^*}(t,\ell) = -t^{-1}\ell^{2} - t\ell^{2} + t^{-1} + t, \qquad L^{2}_{K^*}(t,\ell) = -t^{-1}\ell - t\ell + t^{-1} + t. 
$$
Therefore, knots $K$ and $K^*$ are both non-trivial and non-equivalent to each other. 
}
\end{example}

%%%%%%%%%%%

\section{Behavior under reflection and orientation reversing}
\label{inv}

Now we describe behavior of $L$-polynomial under reflection of a diagram and under changing its orientation. 

 \begin{theorem} \label{th-mirror}
Let $D$ be an oriented virtual knot diagram. Denote by $D^*$ its mirror image and by $D^-$ its reverse. Then for any $n\in \mathbb{N}$ we have 
$$
L^{n}_{D^*}(t,\ell) = -L^{n}_{D}(t^{-1},\ell) \qquad \text{and} \qquad   L^{n}_{D^-}(t,\ell) = L^{n}_{D}(t^{-1},\ell). 
$$
\end{theorem}

\begin{proof} Let $c$ be a classical crossing in $D$. Denote by $c^*$ and $c^-$ be the corresponding crossings in diagrams $D^*$ and $D^-$, respectively. We already mentioned in the proof of Lemma~\ref{lemma1} that  $\operatorname{sgn}(c^-) = \operatorname{sgn}(c )= - \operatorname{sgn} (c^*)$ and $\operatorname{Ind}(c^*) = \operatorname{Ind}(c^-)= -\operatorname{Ind}(c)$. 

Comparing Fig.~\ref{fig5} and Fig.~\ref{fig15new} we see that $D^-_{c^-}$ is equivalent to $D_{c}$, i.~e. $D^-_{c^-} = D_c$,  and $D^{*}_{c^*}$ is inverse of $D_{c}$, i.~e. $D^{*}_{c^*} = (D_c)^-$.    
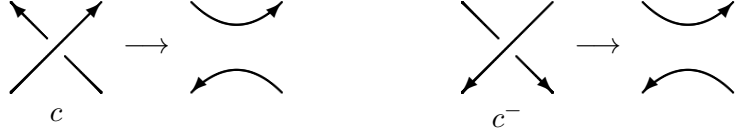
\begin{figure}[!ht]
\centering 
\unitlength=0.6mm
\begin{picture}(0,25)(0,5)
\thicklines
\qbezier(-80,10)(-80,10)(-60,30)
\qbezier(-80,30)(-80,30)(-72,22) 
\qbezier(-60,10)(-60,10)(-68,18)
\put(-75,25){\vector(-1,1){5}}
\put(-65,25){\vector(1,1){5}}
\put(-50,20){\makebox(0,0)[cc]{$\longrightarrow$}}
\qbezier(-40,30)(-30,20)(-20,30)
\qbezier(-40,10)(-30,20)(-20,10) 
\put(-35,15){\vector(-1,-1){5}}
\put(-25,25){\vector(1,1){5}}
\put(-70,5){\makebox(0,0)[cc]{$c$}}
\qbezier(20,10)(20,10)(40,30)
\qbezier(20,30)(20,30)(28,22) 
\qbezier(40,10)(40,10)(32,18)
\put(25,15){\vector(-1,-1){5}}
\put(35,15){\vector(1,-1){5}}
\put(50,20){\makebox(0,0)[cc]{$\longrightarrow$}}
\qbezier(60,30)(70,20)(80,30)
\qbezier(60,10)(70,20)(80,10) 
\put(65,15){\vector(-1,-1){5}}
\put(75,25){\vector(1,1){5}}
\put(30,5){\makebox(0,0)[cc]{$c^-$}}
\end{picture}
\caption{Smoothings against orientation at $c$ and $c^-$.} \label{fig15new}
\end{figure}

\noindent By Lemma~\ref{lemma1} we have $\nabla J_{n}(D_{c}) = \nabla J_{n}(D^{-}_{c^-}) = -\nabla J_{n}(D^*_{c^{*}})$.  
Therefore   
\begin{eqnarray*}
L^{n}_{D^-}(t,\ell) & = & \sum_{c^- \in C(D^-)} \operatorname{sgn}(c^-)\left( t^{\operatorname{Ind}(c^-)} \ell^{\abs{\nabla J_{n}(D^-_{c^-})}} - \ell^{\abs{\nabla J_{n}(D^-)}} \right) \cr  
& = & \sum_{c \in C(D)} \operatorname{sgn}(c) \left( t^{(-\operatorname{Ind}(c))} \ell^{\abs{\nabla J_{n}(D_{c})}} - \ell^{\abs{\nabla J_{n}(D)}} \right) = L^{n}_{D}(t^{-1},\ell) . 
\end{eqnarray*}
Analogously, in the case of the mirror image we get 
\begin{eqnarray*}
L^{n}_{D^*}(t,\ell) & = & \sum_{c^ {*}\in C(D^*)} \operatorname{sgn}(c^*) \left( t^{\operatorname{Ind}(c^*)} \ell^{\abs{\nabla J_{n}(D^{*}_{c^*})}} - \ell^{\abs{\nabla J_{n}(D^*)}} \right) \cr  
& = & \sum_{c\in C(D)} - \operatorname{sgn}(c) \left( t^{(-\operatorname{Ind}(c))} \ell^{\abs{\nabla J_{n}(D_{c})}} - \ell^{\abs{\nabla J_{n}(D)}} \right) = -L^{n}_{D}(t^{-1},\ell) .  
\end{eqnarray*}    
\end{proof}
\section{Cosmetic crossing change conjecture}
\label{cs}

A crossing in a knot diagram is said to be \emph{nugatory} if it can be removed by twisting part of the knot,  see Fig.~\ref{fig20new}.  An example of a nugatory crossing is one that can be undone with an RI-move. Obviously, under applying a crossing change operation at a nugatory crossing we will get a diagram equivalent to the original diagram. Fig.~\ref{fig20new} shows a general form for a nugatory crossing. 
\begin{figure}[!ht]
\centering 
\unitlength=0.6mm
\begin{picture}(0,20)(0,10)
\thicklines
\qbezier(-30,10)(-30,10)(-30,30)
\qbezier(-30,10)(-30,10)(-10,10) 
\qbezier(-10,10)(-10,10)(-10,30)
\qbezier(-30,30)(-30,30)(-10,30)
\put(-20,20){\makebox(0,0)[cc]{$K$}}
\qbezier(30,10)(30,10)(30,30)
\qbezier(30,10)(30,10)(10,10) 
\qbezier(10,10)(10,10)(10,30)
\qbezier(30,30)(30,30)(10,30)
\put(20,20){\makebox(0,0)[cc]{$K'$}}
\qbezier(-10,15)(-10,15)(10,25)
\qbezier(-10,25)(-10,25)(-2,21)
\qbezier(10,15)(10,15)(2,19)
\end{picture}
\caption{Nugatory crossing.} \label{fig20new}
\end{figure}
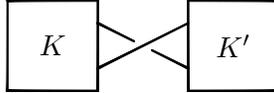

\begin{definition}\cite{folwaczny2013linking} {\rm 
A crossing change in knot diagram $D$ is said to be \emph{cosmetic} if the new diagram, say $D'$, is isotopic (classically or virtually) to $D$.  
}
\end{definition}

A crossing change on nugatory crossing is called \emph{trivial cosmetic crossing change}.  The following question is still open.

\begin{question} \cite[Problem~1.58]{kirby2problems}  {\rm 
Do non-trivial cosmetic crossing change exit? 
}
\end{question}

This question, often referred to as the \emph{cosmetic crossing change conjecture} or the  \emph{nugatory crossing conjecture}, has been answered in the negative for many classes of classical knots, see~\cite{Kalfagianni} and references therein.    
  
For virtual knots this question also has been answered in the negative for a wide class of knot.  In \cite[p.~15]{folwaczny2013linking}, L.~Folwaczny and L.~Kauffman proved that a crossing  $c$ in $D$ with $\operatorname{Ind}(c) \neq 0$ is not cosmetic. 

The following statement gives one more condition on a crossing, with which one can say that the crossing is not a cosmetic.   

\begin{theorem} \label{theorem5.2}
Let $D$ be a virtual knot diagram and $c$ be a crossing in $D$. If $\operatorname{Ind}(c)\neq 0$ or  there exists $n\in \mathcal{N}_D$ such that $ \nabla J_{n}(D_{c}) \neq \pm \nabla J_{n}(D)$, then $c$ is not a cosmetic crossing. 
\end{theorem}

\begin{proof} 
Let $c$ be a classical crossing of $D$ such that either $ \nabla J_{n}(D_{c}) \neq \pm \nabla J_{n}(D)$ or $\text{Ind}(c)\neq 0$. Denote by $D'$ the virtual knot diagram obtained from $D$ by crossing change at $c$. Let $c'$ is the corresponding crossing of $D'$. Then 
$$
\begin{gathered}
L_{D'}^{n}(t,\ell) - \operatorname{sgn}(c') \left(t^{\operatorname{Ind}(c')} \ell^{\abs{\nabla J_{n}(D'_{c'})}}  -  \ell^{\abs{\nabla J_{n}(D')}} \right) \qquad \qquad \qquad \\  \qquad \qquad \qquad
=  L_{D}^{n}(t,\ell) - \operatorname{sgn}(c) \left(t^{\operatorname{Ind}(c)} \ell^{\abs{\nabla J_{n}(D_{c})}}  -  \ell^{\abs{\nabla J_{n}(D)}} \right) . 
\end{gathered}
$$ 

It was shown in the proof of Lemma~\ref{lemma2} that $\operatorname{sgn} (c') = - \operatorname{sgn} (c)$, $\operatorname{Ind} (c') = - \operatorname{Ind} (c)$, and $\nabla J_n (D') = \nabla J_n (D)$.  
Since $D'_{c'}$ is inverse of $D_c$, see Fig.~\ref{fig5}, by Lemma~\ref{lemma2} we have $\nabla J_n (D'_{c'}) = - \nabla J_n (D_c)$. Hence 
$$
L_{D'}^{n}(t,\ell) - L_{D}^{n}(t,\ell) =  \operatorname{sgn}(c) \left( 2 \ell^{\abs{\nabla J_{n}(D)}} - \left( t^{\operatorname{Ind}(c)} + t^{- \operatorname{Ind} (c)} \right) \ell^{\abs{\nabla J_n (D_c)}} \right) .
$$ 
If $\operatorname{Ind} (c) \neq 0$, then $L_{D'}^{n}(t,\ell) \neq  L_{D}^{n}(t,\ell)$. If $\operatorname{Ind} (c) = 0$ and $\abs{\nabla J_n (D_c)} \neq \abs{\nabla J_n (D)}$, then $L_{D'}^{n}(t,\ell) \neq  L_{D}^{n}(t,\ell)$ also. Since $L$-polynomials are virtual knot invariants, $D'$ is not equivalent to $D'$, whence crossing $c$ is non-cosmetic.  
\end{proof} 

\begin{corollary}
If for each classical crossing $c$ of an oriented virtual knot diagram $D$ we have $\operatorname{Ind}(c)\neq 0$ or $\nabla J_{n}(D_{c}) \neq \pm \nabla J_{n}(D)$ for some $n$, then $D$ does not admit cosmetic crossing change.
\end{corollary}

\section{F-polynomials}
\label{F-poly}

We observe that, due to absolute values of dwrithe, in some cases $L$-polynomials fail to distinguish given two virtual knots, see an example presented in Fig.~\ref{fig20}. To resolve this problem, we modify $L$-polynomials and define new polynomials which will be referred as  $F$-polynomials.

\begin{definition}  \label{g-pol} {\rm 
Let $D$ be an oriented virtual knot diagram and $n$ be a positive integer. Then \emph{$n$-th  $F$-polynomial} of $D$ is defined as 
$$
\begin{gathered}
F_{D}^{n}(t,\ell) = \sum_{c \in C(D)} \operatorname{sgn}(c)t^{\text{Ind}(c)} \ell^{\nabla J_{n}(D_{c})}  \qquad \qquad \qquad  \\ 
\qquad \qquad \qquad \qquad -  \sum _{c\in T_{n}(D)} \operatorname{sgn}(c) \ell^{\nabla J_{n}(D_{c})} - \sum _{c\notin T_{n}(D)} \operatorname{sgn} (c) \ell^{\nabla J_{n}(D)}, 
\end{gathered}
$$ 
where  $T_{n}(D)=\{c \in C(D) \mid \nabla J_{n}(D_{c}) = \pm \nabla J_{n}(D)\}$. 
}
\end{definition}

\begin{remark} \label{rem6.2} {\rm 
By replacing $\nabla J_n (D)$ by $\abs{\nabla J_n (D)}$ and  $\nabla J_n (D_c)$ by $\abs{\nabla J_n (D_c)}$ in the definition of $F$-polynomial one will get $L$-polynomial. Hence by replacing each $\ell^q$, $q \in \mathbb Z$, in $F$-polynomial by $\ell^{\abs{q}}$ one will get $L$-polynomial. 
} 
\end{remark}

\begin{example} {\rm 
Consider the virtual knot diagram $D$ presented in Fig.~\ref{fig6}. It has four classical crossings $\{ \alpha, \beta, \gamma. \delta\}$.To calculate $F_D^n(t,\ell)$ we should find $T_n (D)$. Diagrams $D_{\alpha}$, $D_{\beta}$, $D_{\gamma}$ and $D_{\delta}$ are presented in Fig.~\ref{fig7}. It was shown in Example~\ref{ex3.2} that $\nabla J_1 (D) = 0$ and $\nabla J_2 (D) = 0$. Comparing with Table~\ref{t1} we get $T_{1}(D) =  \{\gamma,\delta\}$ and $T_{2}(D) = \{\gamma,\delta\}$. Using values presented in Table~\ref{t1}, we obtain 
$$
F^{1}_{D}(t,\ell)=-t^{2}\ell^{2}-t^{-2}\ell^{-2}+2
$$
and  
$$
F^{2}_{D}(t,\ell)=-t^{2}\ell^{-1} - t^{-2}\ell +2. 
$$
For $n\geq 3$, we have $T_{n}(D)=\{\alpha,\beta,\gamma,\delta\}$ and 
$$
F^{n}_{D}(t,\ell) = P_{D}(t) =  2 - t^{-2} - t^{2}. 
$$ 
} 
\end{example}

\begin{theorem} \label{gth} 
For any $n\in \mathbb{N}$, the polynomial $F^{n}_{K}(t,\ell)$ is a virtual knot invariant. 
\end{theorem}

\begin{proof}
The proof is analogous to the proof of Theorem~\ref{thm-h} and consists of observing the behavior of $F^{n}_{K}(t, \ell)$ under classical Reidemeister moves RI, RII, RIII, and semi virtual move SV. Considering the set $T_{n}(D)=\{c \in C(D) \mid \nabla J_{n}(D_{c}) = \pm \nabla J_{n}(D)\}$, we rewrite $F^{n}_{K} (t, \ell)$ in the form
$$
\begin{gathered}
F_{D}^{n}(t,\ell) = \sum_{c \in T_{n}(D)} \operatorname{sgn}(c) \left[ t^{\text{Ind}(c)} \ell^{\nabla J_{n}(D_{c})}  -  \ell^{\nabla J_{n}(D_{c})} \right] \cr 
\qquad \qquad \qquad + \sum _{c \in C(D) \setminus T_{n}(D)}  \operatorname{sgn}(c) \left[ t^{\text{Ind}(c)} \ell^{\nabla J_{n}(D_{c})}  - \ell^{\nabla J_{n}(D)} \right]. 
\end{gathered}
$$ 

\underline{RI-move:} Let $D'$ be a diagram obtained from $D$ by move RI and $c'$ be the new crossing of $D'$. It can be seen from the proof of Theorem~\ref{thm-h} that $\operatorname{Ind}(c')=0$ and $\abs{\nabla J_{n} (D'_{c'})} = \abs{\nabla J_{n} (D')}$. Hence, $c' \in T_{n}(D')$ and 
$$
F^{n}_{D'} (t, \ell) =  F^{n}_{D} (t, \ell) + \operatorname{sgn}(c') \left[ t^{\text{Ind}(c')} \ell^{\nabla J_{n}(D'_{c'})}  -  \ell^{\nabla J_{n}(D'_{c'})} \right] = F^{n}_{D} (t, \ell). 
$$

\underline{RII-move:} Let $D'$ be a diagram obtained from $D$ by move RII and let $a$ and $b$ be new crossings in $D'$. It can be seen from the proof of Theorem~\ref{thm-h} that $\operatorname{sgn} (a) = - \operatorname{sgn}(b)$ and $\nabla J_{n} (D'_{a}) = \nabla J_{n} (D'_{b})$. Moreover, $a \in T_{n}(D')$ if and only if $b \in T_{n}(D')$. Therefore, $F^{n}_{D'} (t, \ell) = F^{n}_{D} (t, \ell)$.  

\underline{RIII-move:} In notations of the proof of Theorem~\ref{thm-h} we have $\operatorname{sgn} (a') = \operatorname{sgn}(a)$, $\operatorname{sgn}(b') = \operatorname{sgn}(b)$, and $\operatorname{sgn}(c') = \operatorname{sgn}(c)$.  Moreover,  $\nabla J_n (D'_{a'}) = \nabla J_n (D_a)$, $\nabla J_n (D'_{b'}) = \nabla J_n (D_b)$ and $\nabla J_n (D'_{c'}) = \nabla J_n (D_c)$. Hence $F^{n}_{D'} (t, \ell) = F^{n}_{D} (t, \ell)$. 

\underline{SV-move:} In notations of the proof of Theorem~\ref{thm-h} we have $\operatorname{c'} = \operatorname{c}$. Moreover, $\nabla J_{n} (D'_{c'}) = \nabla J_{n} (D_{c})$ and $\nabla J_{n} (D') = \nabla J_{n} (D)$. Therefore, $F^{n}_{D'} (t, \ell) = F^{n}_{D} (t, \ell)$. 
\end{proof}

\begin{proposition}  
If two oriented virtual knots are distinguished by $L$-polynomials, then they are distinguished by $F$-polynomials too. The converse property doesn't hold. 
\end{proposition}
 
\begin{proof}  
Let $K_1$ and $K_2$ be two oriented virtual knot distinguished by $L$-polynomial. Hence there exist $n \in \mathbb N$, $p\in \mathbb Z$, and $q \in \mathbb N$ such that coefficients of $t^p \ell^q$ in $L^n_{K_1}(t,\ell)$  and $L^n_{K_2}(t,\ell)$  are different.

Let $A_1$ and $A_{2}$ be the coefficient of $t^p \ell^q$  and  $t^p \ell^{-q}$ in $F^n_{K_1}(t,\ell)$, respectively. Let $B_1$ and $B_{2}$ be the coefficient of $t^p \ell^q$  and  $t^p \ell^{-q}$ in $F^n_{K_2}(t,\ell)$, respectively. By Remark~\ref{rem6.2}, the coefficient of $t^p \ell^q$ in $L^n_{K_1}(t,\ell)$ is equal $A_1 + A_2$, and the coefficient of $t^p \ell^q$ in  $L^n_{K_2}(t,\ell)$ is equal $B_1 + B_2$. But these coefficients should be different, i.~e.  $A_1+A_2 \neq B_1+ B_2$. Therefore, either $A_1 \neq  B_1$ or $A_2 \neq B_2$. Hence  $F^n_{K_1}(t,\ell) \neq  F^n_{K_2}(t,\ell)$. 

Example~\ref{ex6.5} demonstrate that the converse property doesn't hold.  
\end{proof} 
 
\begin{example} \label{ex6.5} {\rm 
Consider two oriented virtual knots $K$ and $K'$ depicted in Fig.~\ref{fig20}. In Table~\ref{t6} we present the writhe polynomial, $L$-polynomials, and $F$-polynomials, calculated for $K$ and $K'$. One can see that $K$ and $K'$ can not be distinguished by the writhe polynomial and the $L$-polynomials. But $F$-polynomials distinguish them. 
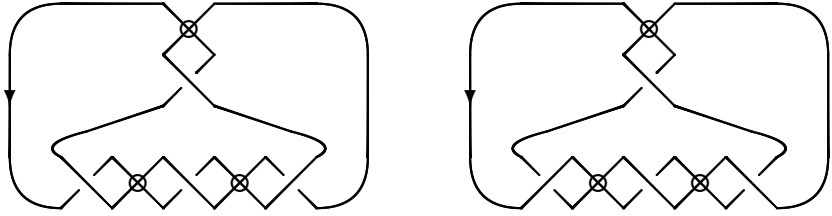
\begin{figure}[!ht]
\centering 
\unitlength=0.34mm
\begin{picture}(0,80)
\put(-90,0){\begin{picture}(0,80) 
\thicklines
\qbezier(-50,20)(-50,20)(-30,0)
\qbezier(-50,0)(-50,0)(-43,7) 
\qbezier(-30,20)(-30,20)(-37,13)
\qbezier(-30,0)(-30,0)(-10,20)
\qbezier(-30,20)(-30,20)(-10,0) 
\put(-20,10){\circle{6}}
\qbezier(-10,20)(-10,20)(10,0)
\qbezier(-10,0)(-10,0)(-3,7)
\qbezier(10,20)(10,20)(3,13) 
\qbezier(10,20)(10,20)(30,0)
\qbezier(10,0)(10,0)(30,20) 
\put(20,10){\circle{6}}
\qbezier(30,0)(30,0)(50,20)
\qbezier(30,20)(30,20)(37,13) 
\qbezier(50,0)(50,0)(43,7)
\qbezier(-10,40)(-10,40)(-3,47)
\qbezier(3,53)(3,53)(10,60) 
\qbezier(-10,60)(-10,60)(10,40)
\qbezier(-10,60)(-10,60)(10,80)
\qbezier(-10,80)(-10,80)(10,60) 
\put(0,70){\circle{6}}
\qbezier(-50,20)(-60,25)(-40,30)
\qbezier(-40,30)(-10,40)(-10,40)
\qbezier(50,20)(60,25)(40,30)
\qbezier(40,30)(10,40)(10,40)
\qbezier(-50,0)(-70,0)(-70,20)
\qbezier(-70,20)(-70,20)(-70,60)
\qbezier(-70,60)(-70,80)(-50,80) 
\qbezier(-50,80)(-50,80)(-10,80) 
\qbezier(50,0)(70,0)(70,20)
\qbezier(70,20)(70,20)(70,60)
\qbezier(70,60)(70,80)(50,80) 
\qbezier(50,80)(50,80)(10,80) 
\put(-70,45){\vector(0,-1){5}}
\end{picture}}
%%%
\put(90,0){\begin{picture}(0,80) 
\thicklines
\qbezier(-50,0)(-50,0)(-30,20)
\qbezier(-50,20)(-50,20)(-43,13) 
\qbezier(-30,0)(-30,0)(-37,7)
\qbezier(-30,20)(-30,20)(-10,0)
\qbezier(-30,0)(-30,0)(-10,20) 
\put(-20,10){\circle{6}}
\qbezier(-10,20)(-10,20)(10,0)
\qbezier(-10,0)(-10,0)(-3,7)
\qbezier(10,20)(10,20)(3,13) 
\qbezier(10,0)(10,0)(30,20)
\qbezier(10,20)(10,20)(30,0) 
\put(20,10){\circle{6}}
\qbezier(30,20)(30,20)(50,0)
\qbezier(30,0)(30,0)(37,7)
\qbezier(50,20)(50,20)(43,13) 
\qbezier(-10,40)(-10,40)(-3,47)
\qbezier(3,53)(3,53)(10,60) 
\qbezier(-10,60)(-10,60)(10,40)
\qbezier(-10,60)(-10,60)(10,80)
\qbezier(-10,80)(-10,80)(10,60) 
\put(0,70){\circle{6}}
\qbezier(-50,20)(-60,25)(-40,30)
\qbezier(-40,30)(-10,40)(-10,40)
\qbezier(50,20)(60,25)(40,30)
\qbezier(40,30)(10,40)(10,40)
\qbezier(-50,0)(-70,0)(-70,20)
\qbezier(-70,20)(-70,20)(-70,60)
\qbezier(-70,60)(-70,80)(-50,80) 
\qbezier(-50,80)(-50,80)(-10,80) 
\qbezier(50,0)(70,0)(70,20)
\qbezier(70,20)(70,20)(70,60)
\qbezier(70,60)(70,80)(50,80) 
\qbezier(50,80)(50,80)(10,80) 
\put(-70,45){\vector(0,-1){5}}
\end{picture}}
\end{picture}
\caption{Oriented virtual knots $K$ (on left) and $K'$ (on right).} \label{fig20}
\end{figure}

\begin{table}[ht]
\caption{Invariants of virtual knots $K$ and $K'$ shown in Fig.~\ref{fig20}.} 
{\small 
\begin{tabular}{ l | l } \hline 
Virtual knot $K$ & Virtual knot $K'$\\  \hline 
\multicolumn{2}{c} {Writhe polynomial} \\  
$W_{K}(t)=-2t^{2}-t^{-2}+1$ & $W_{K'}(t)=-2t^{2}-t^{-2}+1$\\ \hline 
\multicolumn{2}{c} {$L$-polynomials} \\  
$L^{1}_{K}(t,\ell)=t^{-1} \ell^{2}-t\ell^{2}-t-t^{-3}+2\ell^{3}$ & $L^{1}_{K'}(t,\ell)=t^{-1} \ell^{2}-t \ell^{2}-t-t^{-3}+2\ell^{3}$\\
$L^{2}_{K}(t,\ell)=t^{-1}\ell-t\ell-t-t^{-3}+2$ & $L^{2}_{K'}(t,\ell)=t^{-1} \ell - t \ell -t- t^{-3} + 2$\\
$L^{3}_{K}(t,\ell)=t^{-1}-2t-t^{-3}+2\ell$ & $L^{3}_{K'}(t,\ell)=t^{-1}-2t-t^{-3}+2\ell$\\ \hline         
\multicolumn{2}{c} {$F$-polynomials} \\       
$F^{1}_{K}(t,\ell) = t^{-1} \ell^{-2} - t \ell^{-2} - t - t^{-3}+2 \ell^{-3}$ & $F^{1}_{K'}(t,\ell)=t^{-1}\ell^{2}-t \ell^{2}-t-t^{-3}+2\ell^{-3}$\\   
$F^{2}_{K}(t,\ell )= t^{-1}\ell-t\ell-t-t^{-3}+2$ & $F^{2}_{K'}(t,\ell)=t^{-1} \ell^{-1} - t \ell^{-1}- t - t^{-3} + 2$\\
$F^{3}_{K}(t,\ell) = t^{-1}-2t-t^{-3}+2\ell$ & $F^{3}_{K'}(t,\ell)=t^{-1}-2t-t^{-3}+2\ell$\\ \hline   
\end{tabular} 
}\label{t6}
\end{table} } 
\end{example}

%%%%%%%%%
\section{Mutation by positive reflection}
\label{mutbyp}
One of useful local transformation of a knot, producing another knot, is a \emph{mutation} introduced by Conway~\cite{conway1970enumeration}. Conway mutation of a given knot $K$ is achieved by cutting out a tangle (cutting the knot $K$ at four points) and gluing it back  after making a horizontal flip, a vertical flip, or  a $\pi$--rotation. The  mutation  is \emph{positive} if  the  orientation of  the  arcs of the tangle doesn't change under mutation. A \emph{positive reflection} is  a  positive  mutation as shown in Fig.~\ref{fig16new}.  
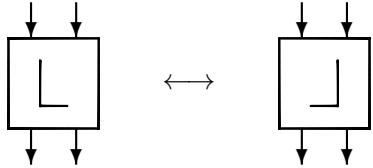
\begin{figure}[!ht]
\centering 
\unitlength=0.6mm
\begin{picture}(0,35)(0,5)
\thicklines
\qbezier(-40,10)(-40,10)(-40,30)
\qbezier(-40,10)(-40,10)(-20,10) 
\qbezier(-20,10)(-20,10)(-20,30)
\qbezier(-40,30)(-40,30)(-20,30)
\put(-35,38){\vector(0,-1){8}}
\put(-25,38){\vector(0,-1){8}}
\put(-35,10){\vector(0,-1){8}}
\put(-25,10){\vector(0,-1){8}}
\qbezier(-33,25)(-33,25)(-33,15)
\qbezier(-33,15)(-33,15)(-27,15)
\put(0,20){\makebox(0,0)[cc]{$\longleftrightarrow$}}
\qbezier(40,10)(40,10)(40,30)
\qbezier(40,10)(40,10)(20,10) 
\qbezier(20,10)(20,10)(20,30)
\qbezier(40,30)(40,30)(20,30)
\put(35,38){\vector(0,-1){8}}
\put(25,38){\vector(0,-1){8}}
\put(35,10){\vector(0,-1){8}}
\put(25,10){\vector(0,-1){8}}
\qbezier(33,25)(33,25)(33,15)
\qbezier(33,15)(33,15)(27,15)
\end{picture}
\caption{Positive reflection mutation.} \label{fig16new}
\end{figure}

It was shown in~\cite{folwaczny2013linking} that the affine index polynomial fails to identify mutation by positive reflection. In Examples~\ref{ex4.2} and~\ref{ex4.3} we construct a family of virtual knots and their positive reflection mutants which are distinguished by $F$-polynomials.  

\begin{example} \label{ex4.2} {\rm 
Consider a virtual knot $K$  and its positive reflection mutant $MK$, obtained by taking positive reflection mutation in the dashed box as shown in~Fig.~\ref{fig17}.   
\begin{figure}[!ht]
\centering 
\unitlength=0.34mm
\begin{picture}(0,125)(0,-5)
\put(-80,0){\begin{picture}(0,120)
\thicklines
\qbezier(10,120)(10,120)(-10,100) 
\qbezier(-10,120)(-10,120)(-3,113)
\qbezier(10,100)(10,100)(3,107)
\qbezier(10,100)(10,100)(-10,80) 
\qbezier(-10,100)(-10,100)(10,80)
\put(0,90){\circle{6}}
\qbezier(10,80)(10,80)(-10,60) 
\qbezier(-10,80)(-10,80)(-3,73)
\qbezier(10,60)(10,60)(3,67)
\qbezier(10,60)(10,60)(30,40)
\qbezier(10,40)(10,40)(30,20)
\qbezier(30,40)(30,40)(23,33)
\qbezier(10,20)(10,20)(17,27)
\qbezier(10,20)(10,20)(30,0)
\qbezier(10,0)(10,0)(30,20)
\put(20,10){\circle{6}}
\qbezier(10,40)(10,40)(-10,30)
\qbezier(10,0)(10,0)(-10,10)
\qbezier(-10,60)(-30,30)(-30,30)
\qbezier(-30,30)(-30,30)(-10,10)
\qbezier(-30,10)(-30,10)(-22,18)
\qbezier(-10,30)(-10,30)(-18,22)
\put(-15,110){\makebox(0,0)[cc]{\footnotesize $a$}}
\put(-15,70){\makebox(0,0)[cc]{\footnotesize $b$}}
\put(30,30){\makebox(0,0)[cc]{\footnotesize $d$}}
\put(-20,10){\makebox(0,0)[cc]{\footnotesize $c$}}
\qbezier(10,120)(40,120)(40,120)
\qbezier(40,120)(50,120)(50,110) 
\qbezier(50,110)(50,110)(50,10) 
\qbezier(50,10)(50,0)(30,0)
\qbezier(-10,120)(-40,120)(-40,120)
\qbezier(-40,120)(-50,120)(-50,110) 
\qbezier(-50,110)(-50,110)(-50,10) 
\qbezier(-50,10)(-50,0)(-40,0)
\qbezier(-40,0)(-40,0)(-30,10)
\put(-50,70){\vector(0,1){5}}
{\thinlines
\multiput(-40,45)(4,0){20}{\line(1,0){2}}
\multiput(-40,-5)(4,0){20}{\line(1,0){2}}
\multiput(-40,45)(0,-4){13}{\line(0,-1){2}}
\multiput(40,45)(0,-4){13}{\line(0,-1){2}}
}
\end{picture}} 
%%%
\put(80,0){\begin{picture}(0,120)
\thicklines
\qbezier(10,120)(10,120)(-10,100) 
\qbezier(-10,120)(-10,120)(-3,113)
\qbezier(10,100)(10,100)(3,107)
\qbezier(10,100)(10,100)(-10,80) 
\qbezier(-10,100)(-10,100)(10,80)
\put(0,90){\circle{6}}
\qbezier(10,80)(10,80)(-10,60) 
\qbezier(-10,80)(-10,80)(-3,73)
\qbezier(10,60)(10,60)(3,67)
\qbezier(-10,60)(-10,60)(-30,40)
\qbezier(-10,20)(-10,20)(-30,40)
\qbezier(-30,20)(-30,20)(-23,27)
\qbezier(-10,40)(-10,40)(-17,33)
\qbezier(-10,20)(-10,20)(-30,0)
\qbezier(-10,0)(-10,0)(-30,20)
\put(-20,10){\circle{6}}
\qbezier(-10,40)(-10,40)(10,30)
\qbezier(-10,0)(-10,0)(10,10)
\qbezier(10,60)(30,30)(30,30)
\qbezier(30,10)(30,10)(10,30)
\qbezier(30,30)(30,30)(22,22)
\qbezier(10,10)(10,10)(18,18)
\put(-15,110){\makebox(0,0)[cc]{\footnotesize $a$}}
\put(-15,70){\makebox(0,0)[cc]{\footnotesize $b$}}
\put(-30,30){\makebox(0,0)[cc]{\footnotesize $d$}}
\put(20,10){\makebox(0,0)[cc]{\footnotesize $c$}}
\qbezier(10,120)(40,120)(40,120)
\qbezier(40,120)(50,120)(50,110) 
\qbezier(50,110)(50,110)(50,10) 
\qbezier(50,10)(50,0)(40,0)
\qbezier(40,0)(40,0)(30,10)
\qbezier(-10,120)(-40,120)(-40,120)
\qbezier(-40,120)(-50,120)(-50,110) 
\qbezier(-50,110)(-50,110)(-50,10) 
\qbezier(-50,10)(-50,0)(-40,0)
\qbezier(-40,0)(-40,0)(-30,0)
\put(-50,70){\vector(0,1){5}}
{\thinlines
\multiput(-40,45)(4,0){20}{\line(1,0){2}}
\multiput(-40,-5)(4,0){20}{\line(1,0){2}}
\multiput(-40,45)(0,-4){13}{\line(0,-1){2}}
\multiput(40,45)(0,-4){13}{\line(0,-1){2}}
}
\end{picture}} 
\end{picture}
\caption{Virtual knot $K$ (left) and its mutant $MK$ (right).} \label{fig17}
\end{figure}
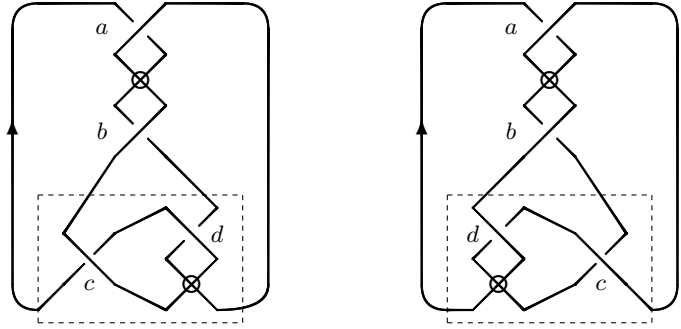

The $F$-polynomials of $K$ and $MK$ are given in Table~\ref{t2}. 

\begin{table}[ht]
\caption{$F$-polynomials of virtual knots $K$ and $MK$ shown in Fig.~\ref{fig17}.} 
{\small 
\begin{tabular}{ l |  l } \hline 
Virtual knot $K$ & Virtual knot $MK$\\ \hline  
$F^{1}_{K}(t,\ell)=t^{2}+1+t^{-1}\ell^{2}-t-\ell^2-\ell^{-2}$ & $F^{1}_{MK}(t,\ell)=t^{2}+1+t^{-1}\ell^{-2}-t-2\ell^{-2}$ \\
$F^{2}_{K}(t,\ell)=t^{2}+1+t^{-1}\ell^{-1}-t-\ell-\ell^{-1}$& $F^{2}_{MK}(t,\ell)=t^{2}+1+t^{-1}\ell-t-2\ell$ \\ \hline  
\end{tabular} } \label{t2} 
\end{table}
}
\end{example}
 
\begin{example} \label{ex4.3} {\rm 
Now we generalize Example~\ref{ex4.2} to an infinite family of virtual knots. Let  $K_{n}$  be a virtual knot obtained from $K$, presented in Fig.~\ref{fig17}, by replacing classical crossing $c$ to $n\geq 1$ classical crossings $c_{1}, \ldots, c_{n}$ as shown in Fig.~\ref{fig18}. Let $MK_n$ be the positive reflection mutant of $K_{n}$. Signs of crossings, index values and  $1$-- and $2$--dwrithes for these knots are presented in Table~\ref{tab3} and Table~\ref{tab4}. Using these data we compute $F$-polynomials of virtual knots $K_n$ and $MK_n$, see Table~\ref{tn}. One can see that $F$-polynomials  distinguish these knots. Therefore, $F$-polynomials can distinguish virtual knots and their positive reflection mutants.     
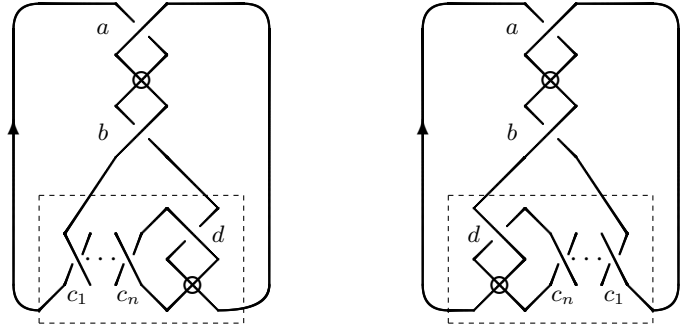
\begin{figure}[!ht]
\centering 
\unitlength=0.34mm
\begin{picture}(0,125)(0,-5)
\put(-80,0){\begin{picture}(0,120)
\thicklines
\qbezier(10,120)(10,120)(-10,100) 
\qbezier(-10,120)(-10,120)(-3,113)
\qbezier(10,100)(10,100)(3,107)
\qbezier(10,100)(10,100)(-10,80) 
\qbezier(-10,100)(-10,100)(10,80)
\put(0,90){\circle{6}}
\qbezier(10,80)(10,80)(-10,60) 
\qbezier(-10,80)(-10,80)(-3,73)
\qbezier(10,60)(10,60)(3,67)
\qbezier(10,60)(10,60)(30,40)
\qbezier(10,40)(10,40)(30,20)
\qbezier(30,40)(30,40)(23,33)
\qbezier(10,20)(10,20)(17,27)
\qbezier(10,20)(10,20)(30,0)
\qbezier(10,0)(10,0)(30,20)
\put(20,10){\circle{6}}
\qbezier(10,40)(10,40)(0,30)
\qbezier(10,0)(10,0)(0,10)
\qbezier(-10,60)(-30,30)(-30,30)
\qbezier(-30,30)(-30,30)(-20,10)
\qbezier(-30,10)(-30,10)(-27,18)
\qbezier(-20,30)(-20,30)(-23,22)
\qbezier(-10,30)(-10,30)(0,10)
\qbezier(-10,10)(-10,10)(-7,18)
\qbezier(0,30)(0,30)(-3,22)
\put(-15,20){\makebox(0,0)[cc]{$\cdots$}}
\put(-15,110){\makebox(0,0)[cc]{\footnotesize $a$}}
\put(-15,70){\makebox(0,0)[cc]{\footnotesize $b$}}
\put(30,30){\makebox(0,0)[cc]{\footnotesize $d$}}
\put(-25,5){\makebox(0,0)[cc]{\footnotesize $c_{1}$}}
\put(-5,5){\makebox(0,0)[cc]{\footnotesize $c_{n}$}}
\qbezier(10,120)(40,120)(40,120)
\qbezier(40,120)(50,120)(50,110) 
\qbezier(50,110)(50,110)(50,10) 
\qbezier(50,10)(50,0)(30,0)
\qbezier(-10,120)(-40,120)(-40,120)
\qbezier(-40,120)(-50,120)(-50,110) 
\qbezier(-50,110)(-50,110)(-50,10) 
\qbezier(-50,10)(-50,0)(-40,0)
\qbezier(-40,0)(-40,0)(-30,10)
\put(-50,70){\vector(0,1){5}}
{\thinlines
\multiput(-40,45)(4,0){20}{\line(1,0){2}}
\multiput(-40,-5)(4,0){20}{\line(1,0){2}}
\multiput(-40,45)(0,-4){13}{\line(0,-1){2}}
\multiput(40,45)(0,-4){13}{\line(0,-1){2}}
}
\end{picture}} 
%%%
\put(80,0){\begin{picture}(0,120)
\thicklines
\qbezier(10,120)(10,120)(-10,100) 
\qbezier(-10,120)(-10,120)(-3,113)
\qbezier(10,100)(10,100)(3,107)
\qbezier(10,100)(10,100)(-10,80) 
\qbezier(-10,100)(-10,100)(10,80)
\put(0,90){\circle{6}}
\qbezier(10,80)(10,80)(-10,60) 
\qbezier(-10,80)(-10,80)(-3,73)
\qbezier(10,60)(10,60)(3,67)
\qbezier(-10,60)(-10,60)(-30,40)
\qbezier(-10,20)(-10,20)(-30,40)
\qbezier(-30,20)(-30,20)(-23,27)
\qbezier(-10,40)(-10,40)(-17,33)
\qbezier(-10,20)(-10,20)(-30,0)
\qbezier(-10,0)(-10,0)(-30,20)
\put(-20,10){\circle{6}}
\qbezier(-10,40)(-10,40)(0,30)
\qbezier(-10,0)(-10,0)(0,10)
\qbezier(10,60)(30,30)(30,30)
\qbezier(30,10)(30,10)(20,30)
\qbezier(30,30)(30,30)(27,22)
\qbezier(20,10)(20,10)(23,18)
\qbezier(10,10)(10,10)(0,30)
\qbezier(10,30)(10,30)(7,22)
\qbezier(0,10)(0,10)(3,18)
\put(15,20){\makebox(0,0)[cc]{$\cdots$}}
\put(-15,110){\makebox(0,0)[cc]{\footnotesize $a$}}
\put(-15,70){\makebox(0,0)[cc]{\footnotesize $b$}}
\put(-30,30){\makebox(0,0)[cc]{\footnotesize $d$}}
\put(25,5){\makebox(0,0)[cc]{\footnotesize $c_{1}$}}
\put(5,5){\makebox(0,0)[cc]{\footnotesize $c_{n}$}}
\qbezier(10,120)(40,120)(40,120)
\qbezier(40,120)(50,120)(50,110) 
\qbezier(50,110)(50,110)(50,10) 
\qbezier(50,10)(50,0)(40,0)
\qbezier(40,0)(40,0)(30,10)
\qbezier(-10,120)(-40,120)(-40,120)
\qbezier(-40,120)(-50,120)(-50,110) 
\qbezier(-50,110)(-50,110)(-50,10) 
\qbezier(-50,10)(-50,0)(-40,0)
\qbezier(-40,0)(-40,0)(-30,0)
\put(-50,70){\vector(0,1){5}}
{\thinlines
\multiput(-40,45)(4,0){20}{\line(1,0){2}}
\multiput(-40,-5)(4,0){20}{\line(1,0){2}}
\multiput(-40,45)(0,-4){13}{\line(0,-1){2}}
\multiput(40,45)(0,-4){13}{\line(0,-1){2}}
}
\end{picture}} 
\end{picture}
\caption{Virtual knot $K_{n}$ (left) and its mutant $MK_{n}$ (right).} \label{fig18}
\end{figure}

\begin{table}[ht]
\caption{Calculations for virtual knot $K_{n}$, presented in Fig.~\ref{fig18}.}
{\tiny \begin{tabular}{l|r|r|r|r} \hline 
& \multicolumn{2}{c|}{sign} & \multicolumn{2}{c}{index} \\ \hline
 & $n$-even & $n$-odd & $n$-even & $n$-odd \\ \hline  
$a$ & $1$ & $ 1 $ & $0$ & $2$ \\
$b$ & $1$ & $ 1 $ & $-2$ & $0$ \\
$d$ & $-1$ & $ 1 $ & $-2$ & $-1$ \\
$c_{2i}$ & $-1$ & $ -1 $ & $1$ & $-1$ \\
$c_{2i+1}$ & $-1$ & $ -1 $ & $-1$ & $1$ \\ \hline 
\end{tabular}}
\quad 
{\tiny  
\begin{tabular}{l|r|r|r|r} \hline 
& \multicolumn{2}{c}{$1$-dwrithe} & \multicolumn{2}{c}{$2$-dwrithe} \\ \hline 
 & $n$-even & $n$-odd & $n$-even & $n$-odd \\ \hline  
$D$ & $0$ & $ -2$ & $0$ & $1$ \\
$D_{a}$ & $0$ & $ 0 $ & $0$ & $0$ \\
$D_{b}$ & $0$ & $ 0$ & $0$ & $0$ \\
$D_{d}$ & $0$ & $ 2$ & $0$ & $-1$ \\
$D_{c_{2i}}$ & $-2$ & $ 0$ & $1$ & $0$ \\
$D_{c_{2i+1}}$ & $2$ & $ 0$ & $-1$ & $0$ \\ \hline 
\end{tabular}} \label{tab3}
\end{table}

\begin{table}[ht]
\caption{Calculations for virtual knot $MK_{n}$, presented in Fig.~\ref{fig18}.}
{\tiny  \begin{tabular}{l|r|r|r|r} \hline  
& \multicolumn{2}{c|}{sign} & \multicolumn{2}{c}{index} \\ \hline 
 & $n$-even & $n$-odd & $n$-even & $n$-odd \\ \hline  
$a$ & $1$ & $ 1 $ & $0$ & $2$ \\
$b$ & $1$ & $ 1 $ & $-2$ & $0$ \\
$d$ & $-1$ & $ 1 $ & $-2$ & $-1$ \\
$c_{2i}$ & $-1$ & $ -1 $ & $1$ & $-1$ \\
$c_{2i+1}$ & $-1$ & $ -1 $ & $-1$ & $1$ \\ \hline 
\end{tabular}
}
\quad 
{\tiny 
\begin{tabular}{l|r|r|r|r} \hline 
& \multicolumn{2}{c}{$1$-dwrithe} & \multicolumn{2}{c}{$2$-dwrithe} \\ \hline 
 & $n$-even & $n$-odd & $n$-even & $n$-odd \\ \hline  
$D$ & $0$ & $ -2$ & $0$ & $1$ \\
$D_{a}$ & $0$ & $ 0 $ & $0$ & $0$ \\
$D_{b}$ & $0$ & $ 0$ & $0$ & $0$ \\
$D_{d}$ & $0$ & $ -2$ & $0$ & $1$ \\
$D_{c_{2i}}$ & $2$ & $ 0$ & $-1$ & $0$ \\
$D_{c_{2i+1}}$ & $-2$ & $ 0$ & $1$ & $0$ \\ \hline 
\end{tabular}
}  \label{tab4}
\end{table}

\begin{table}[ht]
\caption{$F$-polynomials of virtual knots $K_n$ and $MK_n$.}
{\small \begin{tabular}{| l |} \hline
for $n$ odd: \\ \hline 
$F_{K_{n}}^{1}(t,\ell) = t^{2}+1+t^{-1}\ell^{2}-\frac{n+1}{2}t-\frac{n-1}{2}t^{-1}-(2-n)\ell^{-2}-\ell^2$ \\ 
$F_{K_{n}}^{2}(t,\ell) = t^{2}+1+t^{-1}\ell^{-1}-\frac{n+1}{2}t-\frac{n-1}{2}t^{-1}-(2-n)\ell-\ell^{-1}$ \\ \hline   
$F_{MK_{n}}^{1}(t,\ell) = t^{2}+1+t^{-1}\ell^{-2}-\frac{n+1}{2}t-\frac{n-1}{2}t^{-1}-(3-n)\ell^{-2}$ \\  
$F_{MK_{n}}^{2}(t,\ell) = t^{2}+1+t^{-1}\ell-\frac{n+1}{2}t-\frac{n-1}{2}t^{-1}-(3-n)\ell$ \\ \hline  
for $n$ even:\\ \hline     
$F_{K_{n}}^{1}(t,\ell) = n-\frac{n}{2}(t^{-1}\ell^2+t\ell^{-2})$\\   
$F_{K_{n}}^{2}(t,\ell) =  n-\frac{n}{2}(t^{-1}\ell^{-1}+t\ell)$ \\ \hline            
$F_{MK_{n}}^{1}(t,\ell)  = n-\frac{n}{2}(t^{-1}\ell^{-2}+t\ell^{2})$\\             
$F_{MK_{n}}^{2}(t,\ell)  = n-\frac{n}{2}(t^{-1}\ell+t\ell^{-1})$\\  \hline         
\end{tabular}} \label{tn}
\end{table}
}
\end{example}

%%%
\end{document}